\documentclass[11pt]{amsart}

\usepackage{amsmath,amssymb,amsthm,enumerate}
\usepackage{thmtools,csquotes}
\usepackage[dvipsnames]{xcolor}
\usepackage{comment}

\usepackage{geometry}
\usepackage{todonotes}

\usepackage{tikz-cd}
\usepackage{hyperref}
\hypersetup{
	colorlinks=true,
	linkcolor=red,
	citecolor=OliveGreen}

\theoremstyle{plain}

\newtheorem{thm}{Theorem}[section]
\newtheorem{prop}[thm]{Proposition}
\newtheorem{lem}[thm]{Lemma}
\newtheorem{cor}[thm]{Corollary}
\newtheorem{claim}[thm]{Claim}
\newtheorem*{TermConj}{Termination of Flips Conjecture}

\theoremstyle{definition}

\newtheorem{dfn}[thm]{Definition}
\newtheorem{rem}[thm]{Remark}

\newcommand{\Z}{\mathbb{Z}}

\newcommand{\Q}{\mathbb{Q}}
\newcommand{\Qq}{\mathbb{Q}}
\newcommand{\R}{\mathbb{R}}
\newcommand{\Rr}{\mathbb{R}}
\newcommand{\C}{\mathbb{C}}

\newcommand{\dto}{\dashrightarrow}
\newcommand{\xto}{\xrightarrow{f}}
\newcommand{\lf}{\lfloor}
\newcommand{\rf}{\rfloor}

\DeclareMathOperator{\codim}{codim}
\DeclareMathOperator{\Exc}{Exc}
\DeclareMathOperator{\mult}{mult}
\DeclareMathOperator{\Supp}{Supp}
\DeclareMathOperator{\nklt}{nklt}

\DeclareMathOperator{\Center}{center}

\numberwithin{equation}{section}

\begin{document}
	
	\title[On the Termination of Flips]{On the Termination of Flips for Log Canonical Generalized Pairs}
	
	\author{Guodu Chen and Nikolaos Tsakanikas}
	
	\address{Institute for Theoretical Sciences, Westlake Institute for Advanced Study, Westlake University, Hangzhou, Zhejiang, 310024, China}
	\email{chenguodu@westlake.edu.cn}
	
	\address{Fachrichtung Mathematik, Campus, E2.4, Universit\"at des Saarlandes, 66123 Saarbr\"ucken, Germany}
	\email{tsakanikas@math.uni-sb.de}
	

	\begin{abstract}
		We prove the termination of flips for pseudo-effective NQC log canonical generalized pairs of dimension $4$. As main ingredients, we verify the termination of flips for $3$-dimensional NQC log canonical generalized pairs, and show that the termination of flips for pseudo-effective NQC log canonical generalized pairs which admit NQC weak Zariski decompositions follows from the termination of flips in lower dimensions.
	\end{abstract}
	
	\maketitle
	\begingroup
	\hypersetup{linkcolor=black}
	\setcounter{tocdepth}{1}
	\tableofcontents
	\endgroup

	\section{Introduction}
	\label{section:Intro}
	
	Throughout the paper we work over the field $\C$ of complex numbers. See Section \ref{section:prelim} for notation and terminology.
	
	The concept of a generalized pair (g-pair for short) was introduced by Birkar and Zhang in \cite{BZ16} and generalizes the notion of a usual pair. Roughly speaking, a g-pair is a couple of the form $(X,B+M)$, where $X$ is a normal projective variety, $B$ is an effective $\R$-divisor on $X$ and the $\R$-divisor $M$ has certain positivity properties, such that $K_X+B+M$ is $\R$-Cartier. Generalized pairs underlie many of the latest developments in birational geometry, the majority of which is outlined in \cite{Bir20a}.
	It is worthwhile to mention here that the applications of g-pairs range from the celebrated solution of the Borisov-Alexeev-Borisov conjecture \cite{Bir21} and the boundedness of complements conjecture \cite{Bir19,HLS19,Chen20,ChenXue20} to Fujita's spectrum conjecture \cite{HanLi20,HanLi21,Li21} and the so-called connectedness principle of Shokurov--Koll{\'a}r \cite{HH19,Bir20b,FS20}. For even further applications of g-pairs we refer to \cite{Liu18,HanLiu19,HanLiu20,Has20}, just to name a few.
	
	In fact, not only have g-pairs been implemented successfully in a wide range of contexts as is apparent from the above, but also several recent papers, such as \cite{Mor18,HanLi,LT19,HM20}, indicate that it is actually essential to understand their birational geometry, even if one is only interested in studying the birational geometry of varieties and investigating the main open problems in the Minimal Model Program (MMP) for usual pairs. Consequently, it is very natural to formulate and address problems regarding usual pairs in the setting of g-pairs. In doing so, one usually works with so-called \emph{NQC} g-pairs, i.e., g-pairs $ (X,B+M) $ with data $ X' \to X \to Z $ and $ M' $ such that the divisor $M'$ on $X'$ is a non-negative linear combination of $\Q$-Cartier divisors on $X'$ which are nef over $ Z $, since they behave well in the MMP \cite{HanLi}.
	
	In this paper we deal with the termination of flips conjecture for NQC lc g-pairs, cf.\ \cite[Conjecture 3.4]{HanLi}.
	
	\begin{TermConj}
		Let $(X/Z,B+M)$ be an NQC lc g-pair. Then any sequence of flips over $Z$ starting from $(X/Z,B+M)$ terminates.
	\end{TermConj}
	
	We recall briefly the overall progress towards the solution of this conjecture in the context of usual pairs. In dimension $ 3 $, the termination of flips was established for terminal varieties in \cite{Sho85} and for lc pairs in \cite{Kaw92,Sho96}. In higher dimensions, and even in dimension $ 4 $, the termination of flips is an extremely difficult problem. Even though it was settled for terminal $ 4 $-folds already in \cite{KMM87}, the termination of flips for a large class of $ 4 $-dimensional klt pairs was proved approximately $ 20 $ years later in \cite{AHK07} and \cite{Bir10b}, while the case of $ 4 $-dimensional canonical pairs with rational coefficients was treated in \cite{Fuj04,Fuj05}. Moreover, it was shown in \cite{Sho04} that the termination of flips (which is a global statement) follows from two other conjectures (of local nature) concerning singularities of pairs, the \emph{ascending chain condition (ACC) for minimal log discrepancies} and the \emph{lower semi-continuity (LSC) for minimal log discrepancies}.
	
	Finally, we recall that the so-called termination of flips with scaling for $4$-dimensional pairs with $\Q$-factorial and (at worst) dlt singularities was established in \cite{Sho09,Bir09,Bir10a}.
	
	\medskip
	
	The main goal of this paper is to confirm the above conjecture in the pseudo-effective case in dimension $ 4 $.
	
	\begin{thm}\label{thm:terminationforfourfolds}
		Let $(X/Z,B+M)$ be a $4$-dimensional NQC lc g-pair. If $K_X+B+M$ is pseudo-effective over $Z$, then any sequence of flips over $Z$ starting from $(X/Z,B+M)$ terminates.
	\end{thm}
	
	For the proof of the above theorem we proceed as follows. First, by utilizing several ideas from \cite{Kol92,Kaw92,Sho96}, we give a detailed proof of the termination of flips conjecture in dimension $ 3 $.
	
	\begin{thm}\label{thm:terminationforthreefolds}
		Let $(X/Z,B+M)$ be a $3$-dimensional NQC lc g-pair. Then any sequence of flips over $Z$ starting from $(X/Z,B+M)$ terminates.
	\end{thm}
	
	Next, by following the strategy of \cite{Bir07} and by invoking the \emph{special termination for NQC $\Q$-factorial dlt g-pairs} and the \emph{ACC for lc thresholds for g-pairs}, we extend Birkar's termination result \cite[Theorem 1.3]{Bir07} to the context of g-pairs.
	
	\begin{thm} \label{thm:mainconsequence_g}
		Assume the termination of flips for $\Q$-factorial NQC klt g-pairs of dimension at most $ d-1 $.
		
		Let $ (X/Z,B+M) $ be a pseudo-effective NQC lc g-pair of dimension $ d $. If $ (X/Z,B+M) $ admits an NQC weak Zariski decomposition over $Z$, 
		then any sequence of flips over $Z$ starting from $ (X/Z,B+M) $ terminates.
	\end{thm}
	
	We refer to \cite{HM20} for the rational coefficients case of the above theorem. Besides, it is also worthwhile to mention that the paper \cite{HanLi} proved the termination of any MMP with scaling of an ample divisor with respect to any NQC $\Q$-factorial dlt g-pair which admits an NQC weak Zariski decomposition, assuming the existence of NQC weak Zariski decompositions in lower dimensions. Note that this result was refined in \cite{LT19}, namely, it was shown that the previous statement holds assuming instead the existence of minimal models for smooth varieties in lower dimensions. In particular, we conclude that if $(X/Z,B+M)$ is a $5$-dimensional NQC $\Q$-factorial dlt g-pair which admits an NQC weak Zariski decomposition over $Z$, then any $(K_X+B+M)$-MMP with scaling of an ample divisor over $Z$ terminates.
	
	In conclusion, Theorem \ref{thm:terminationforfourfolds} is an immediate consequence of Theorems \ref{thm:terminationforthreefolds} and \ref{thm:mainconsequence_g}. Indeed, the existence of NQC weak Zariski decompositions for $4$-dimensional pseudo-effective NQC lc g-pairs follows from \cite[Theorem E]{LT19} and from the existence of minimal models for smooth 4-folds \cite[Theorem 5-1-15]{KMM87}.
	
	In particular, we deduce the termination of flips for pseudo-effective lc pairs of dimension $ 4 $ with real coefficients as a special case of Theorem \ref{thm:terminationforfourfolds}. The rational coefficients case of the aforementioned termination was announced previously by Moraga in (the first and second version of his preprint) \cite{Mor18}. However, our approach is significantly different from Moraga's, as he builds on ideas from \cite{AHK07} and uses crucially a version of the ACC for minimal log discrepancies for g-pairs in dimension $ 2 $ which is still unknown.
	
	Last but not least, the conclusion of Theorem \ref{thm:mainconsequence_g} can be interpreted roughly as follows: for pseudo-effective g-pairs, if \emph{some} MMP terminates, then \emph{every} MMP terminates. On the other hand, the termination of flips conjecture for non-pseudo-effective g-pairs remains unresolved, even in dimension $ 4 $. It is essentially the only obstacle that one needs to overcome in order to derive the termination of flips just from the existence of minimal models and Mori fiber spaces. We hope to address this problem in a future work.
	
	\medskip
	
	\textit{Structure of the paper}: We outline the organization of the paper. In Section \ref{section:prelim} we recall various definitions and prove certain basic results. In Section \ref{section:term_in_dim_3} we prove Theorem \ref{thm:terminationforthreefolds}, while in Section \ref{section:term_for_psef} we prove Theorems \ref{thm:terminationforfourfolds} and \ref{thm:mainconsequence_g}.
	
	\medskip
	
	\textit{Postscript}: After we posted this preprint on the arXiv we were informed by Joaqu\'in Moraga that in the third version of his preprint \cite{Mor18} he has proved Theorem \ref{thm:terminationforfourfolds} independently using similar methods.

	\medskip
	
	\textit{Acknowledgements}:
	Part of this work was done while G.C.\ visited Chen Jiang at Fudan University, he would like to thank their hospitality. G.C.\ would like to thank his supervisor, Chenyang Xu, for his constant support and suggestions. N.T.\ would like to thank his supervisor, Vladimir Lazi\'c, for guidance and encouragement as well as for numerous valuable discussions and comments. Special thanks go to Jingjun Han for many helpful discussions related to this work, for encouragement and lecture notes. We would also like thank Chen Jiang, Jihao Liu, Yuchen Liu, Yujie Luo, Qingyuan Xue and Sokratis Zikas for useful discussions. Finally, we are grateful to the referees for many valuable comments and suggestions.

	\section{Preliminaries}
	\label{section:prelim}
	
	For an $\R$-divisor $D = \sum d_i D_i$, where the $D_i$ are distinct prime divisors, we set
	$$ 
	D^{<1} := \sum\limits_{i \, : \, d_i < 1} d_i D_i \quad \text{and} \quad 
	D^{=1} := \sum\limits_{i \, : \, d_i = 1} D_i . $$
	The notation $ D \in \Gamma $ means that the coefficients $ d_i $ of $ D $ belong to a set $ \Gamma $.
	
	Throughout the paper we assume that varieties are normal and quasi-projective and that a variety $X$ over a variety $Z$, denoted by $ X/Z $, is projective over $Z$.
	
	Let $ \pi \colon X \to Z $ be a projective morphism of normal varieties and let $ D, D_1, D_2 $ be $ \R $-Cartier $ \R $-divisors on $ X $. We say that $ D $ is \emph{pseudo-effective over $ Z $} if it is pseudo-effective on a very general fiber of $ \pi $. Moreover, $ D_1 $ and $ D_2 $ are said to be \emph{numerically equivalent over $ Z $}, denoted by $ D_1 \equiv_Z D_2 $, if $ D_1\cdot C = D_2 \cdot C $ for any curve $ C $ contained in a fiber of $ \pi $. Finally, we say that an $ \R $-divisor $ P $ on $ X $ is \emph{NQC} if it is a non-negative linear combination of $\Q$-Cartier $\Q$-divisors on $X$ which are nef over $ Z $, see \cite[\S 2.4]{HanLi}; the acronym NQC stands for \emph{nef $\Q$-Cartier combinations}.
	
	\medskip
	
	We recall the notion of an NQC weak Zariski decomposition, which plays a significant role in the paper, and we refer to \cite[Section 2]{HanLi} and \cite[Section 2]{LT19} for further information.
	
	\begin{dfn}\label{dfn:wzd}
		Let $  X \to Z $ be a projective morphism of normal quasi-projective varieties and let $ D $ be an $ \R $-Cartier $\R$-divisor on $ X $. An \emph{NQC weak Zariski decomposition over $Z$} of $ D $ consists of a projective birational morphism $ f \colon W \to X $ from a normal variety $ W $ and a numerical equivalence $ f^* D \equiv_Z P + N $ such that $ P $ is an NQC divisor on $ W $ and $ N $ is an effective $\R$-Cartier $\R$-divisor on $W$. 
	\end{dfn}

	\subsection{Generalized pairs}
	
	For the definitions and the basic results on the singularities of pairs and the Minimal Model Program we refer to \cite{KM98,Kol13,Fuj17}. Here, we discuss very briefly analogous concepts for generalized pairs, while for further details we refer to \cite[Section 4]{BZ16} and \cite[Sections 2 and 3]{HanLi}. 
	
	\begin{dfn}
		A \emph{generalized pair}, abbreviated as \emph{g-pair}, consists of 
		\begin{itemize}
			\item a normal variety $ X $, equipped with projective morphisms 
			$$ X' \overset{f}{\longrightarrow} X \longrightarrow Z , $$ 
			where $ f $ is birational and $ X' $ is a normal variety,
			
			\item an effective $ \R $-divisor $ B $ on $X$, and
			
			\item an $\R$-Cartier $\R$-divisor $ M' $ on $X'$ which is nef over $Z$,
		\end{itemize}
		such that $ K_X + B + M $ is $ \R $-Cartier, where $ M := f_* M' $.
		
		We usually refer to a g-pair as above by saying that $ (X,B+M) $ is a g-pair with data $ X' \overset{f}{\to} X \to Z $ and $ M' $. However, we often denote a g-pair simply by $(X/Z,B+M)$, but remember the whole g-pair structure. Note also that the definition is flexible with respect to $X'$ and $M'$; namely, if $ g \colon W \to X' $ is a projective birational morphism from a normal variety $W$, then we may replace $X'$ with $W$ and $M'$ with $ g^*M'$. Therefore, $X'$ may always be chosen as a sufficiently high birational model of $ X $.
	\end{dfn}
	
	Given a g-pair $ (X,B+M) $ with data $ X' \overset{f}{\to} X \to Z $ and $ M' $, we say that:
	\begin{enumerate}
		\item[(a)] it is an \emph{NQC g-pair} if $M'$ is an NQC divisor on $ X' $,
		
		\item[(b)] it is \emph{pseudo-effective over $Z$} if the divisor $ K_X + B + M $ is pseudo-effective over $Z$, and 
		
		\item[(c)] it \emph{admits an NQC weak Zariski decomposition over $Z$} if the divisor $ K_X + B + M $ admits an NQC weak Zariski decomposition over $Z$.
	\end{enumerate}
	
	\begin{dfn}
		Let $ (X,B+M) $ be a g-pair with data $ X' \overset{f}{\to} X \to Z $ and $ M' $. We may write 
		$$ K_{X'} + B' + M' \sim_\R f^* ( K_X + B + M ) $$
		for some $ \R $-divisor $ B' $ on $ X' $.
		Let $ E $ be a divisorial valuation over $ X $. We may assume that $ E $ is a prime divisor on $ X' $. The \emph{log discrepancy of $ E $} with respect to $ (X,B+M) $ is defined as 
		$$ a(E,X,B+M) :=1-\mult_E B' . $$
		
		The g-pair $ (X,B+M) $ is called:
		\begin{enumerate}
			\item[(a)] \emph{terminal} if $  a(E,X,B+M) > 1 $ for any exceptional divisorial valuation $ E $ over $ X $,
			
			\item[(b)] \emph{klt} if $  a(E,X,B+M) > 0 $ for any divisorial valuation $ E $ over $ X $, and
			
			\item[(c)] \emph{lc} if $  a(E,X,B+M) \geq 0 $ for any divisorial valuation $ E $ over $ X $.
		\end{enumerate}
	\end{dfn}
	
	In the notation from the above definition, we remark that if $ f $ is a log resolution of $ (X,B) $, then the g-pair $ (X,B+M) $ is klt (resp.\ lc) if and only if the coefficients of $ B' $ are $ <1 $ (resp.\ $ \leq 1 $). Additionally, if $ (X,B+M) $ is terminal (resp.\ klt, lc) and $ X $ is $ \Q $-factorial, then $ (X,B) $ is terminal (resp.\ klt, lc) by \cite[Remark 4.2(3)]{BZ16}; in particular, if $ (X,B+M) $ is a $ \Q $-factorial terminal g-pair, then $ \lfloor B \rfloor = 0 $ and $ X $ is smooth in codimension two by \cite[Theorem 4.5 and Corollary 5.18]{KM98}.
	
	\begin{dfn}
		Let $ (X,B+M) $ be an lc g-pair with data $ X' \overset{f}{\to} X \to Z $ and $ M' $. Let $ P' $ be an $ \R $-divisor on $ X' $ which is nef over $ Z $ and let $ N $ be an effective $ \R $-divisor on $ X $ such that $ P+N $ is $ \R $-Cartier, where $ P := f_* P' $. \emph{The lc threshold of $ P+N $ with respect to $ (X,B+M) $} is defined as
		$$ \sup \left\{ t \geq 0 \mid \big(X, (B+tN) + (M+tP) \big) \text{ is lc} \right\} . $$
	\end{dfn}
	
	Note that the lc threshold as defined above might be $+\infty$: this happens when $N=0$ and $P' = f^* P$, see \cite[Section 4]{BZ16}.
	
	\begin{dfn}
		A g-pair $ (X,B+M) $ is called \emph{dlt} if it is lc, if there exists a closed subset $V \subseteq X$ such that the pair $(X \setminus V,B|_{X \setminus V})$ is log smooth, and if $ a(E,X,B+M) = 0$ for some divisorial valuation $ E $ over $X$, then $\Center_X E \nsubseteq V $ and $ \Center_X E \setminus V $ is an lc center of $(X \setminus V,B|_{X \setminus V})$.
	\end{dfn}
	
	Note that we adopted \cite[Definition 2.2]{HanLi} as our definition of dlt singularities, cf.\ \cite[\S 2.13(2)]{Bir19}, \cite[Definition 2.5]{Has20}.
	
	\medskip
	
	Let $ (X,B+M) $ be an lc g-pair. A subvariety $ S $ of $ X $ is called an \emph{lc center} of $ X $ if there exists a divisorial valuation $ E $ over $ X $ such that $  a(E,X,B+M) = 0 $ and $ \Center_X E = S $. The union of all lc centers of $ X $ is denoted by $ \nklt(X,B+M) $ and is called the \emph{non-klt locus} of $ (X,B+M) $.
	
	In the following remark we gather some properties of the non-klt locus of a g-pair that will be used in the proof of Lemma \ref{lem:supplement_g}.
	
	\begin{rem} ~
		\begin{enumerate}[(1)]
			\item If $ (X,B+M) $ is a $ \Q $-factorial dlt g-pair, then by definition and by \cite[Remark 4.2(3)]{BZ16} the underlying pair $ (X,B) $ is also $ \Q $-factorial dlt and the lc centers of $ (X,B+M) $ coincide with those of $ (X,B) $. In particular, by \cite[Theorem 4.16]{Kol13} we obtain
			\[ \nklt(X,B+M) = \nklt(X,B) = \Supp \lfloor B \rfloor . \]
			
			\item If $ (X,B+M) $ is an lc g-pair and if $ h \colon (Y,\Delta+\Xi) \to (X,B+M) $ is a dlt blow-up of $ (X,B+M) $ (see Lemma \ref{lem:modifications}(ii)), then by (1) we obtain
			\[ \nklt(X,B+M) = h \big(\nklt(Y,\Delta+\Xi)\big) = h \big( \Supp \lfloor \Delta \rfloor \big) . \]
		\end{enumerate}
	\end{rem}
	
	We prove below analogs of \cite[Corollary 2.35(1), Proposition 2.36(2) and Lemma 2.29]{KM98}, respectively, in the context of g-pairs.
	
	\begin{lem}\label{lem:singularities_g}
		Let $\big(X,(B+N)+(M+P)\big)$ be a $ \Q $-factorial klt (resp.\ lc) g-pair with data $ X' \overset{f}{\to} X \to Z $ and $ M' + P' $. Then the g-pair $(X,B+M)$ is also klt (resp.\ lc).
	\end{lem}
	
	\begin{proof}
		We may assume that $ f $ is a log resolution of $(X,B+N)$ and we may write
		$$ K_{X'} + B' + M' \sim_\R f^*(K_X + B + M) $$
		for some $ \R $-divisor $ B' $ on $ X' $. We may also write 
		$$ f^* P = P' + E_1  \ \text{ and } \ f^* N = f_*^{-1} N + E_2 , $$
		where $ E_1 $ is an effective $ f $-exceptional $\R$-divisor by the Negativity lemma \cite[Lemma 3.39]{KM98}, as $ P' $ is $ f $-nef, and $ E_2 $ is an effective $ f $-exceptional $\R$-divisor, as $ N \geq 0 $. Therefore,
		$$ K_{X'} + (B' + f_*^{-1} N + E_1 + E_2) + (M' + P') \sim_\R f^* \big(K_X + (B + N) + (M + P) \big) . $$
		Since $\big(X,(B+N)+(M+P)\big)$ is klt (resp.\ lc), the coefficients of the divisor 
		$ B' + f_*^{-1} N + E_1 + E_2 $ are $<1$ (resp.\ $\le1$), whence the statement.
	\end{proof}
	
	\begin{prop} \label{prop:klt_discrep}
		Let $(X,B+M)$ be a klt g-pair with data $X' \xto X\to Z$ and $M'$. Then there exists a positive real number $\varepsilon$ such that there are only finitely many exceptional divisorial valuations $E$ over $ X $ such that $ a(E,X,B+M)<1+\varepsilon$.
		
		In particular, if $(X,B+M)$ is terminal, then we may take $\epsilon:=1-\max_i{b_i}$, where $b_i$ are the coefficients of $B$.
	\end{prop}
	
	\begin{proof}
		We may assume that $f$ is a log resolution of $(X,B)$ such that if we write
		$$ K_{X'} + B' + M' \sim_\R f^*(K_X+B+M) $$
		for some $ \R $-divisor $ B' = (B')^+ - (B')^- $ on $ X' $, then $ \Supp (B')^+ $ is smooth, where the divisors $ (B')^+ $ and $ (B')^- $ are effective and have no common components, cf.\ \cite[Proposition 2.36(1)]{KM98}. Set $\varepsilon := 1-\max \{ \alpha \in \mathcal{C} \}$, where $\mathcal{C}$ is the set of coefficients of $(B')^+$, and note that $ \varepsilon > 0 $, since $ (X,B+M) $ is klt. If $F$ is an exceptional divisorial valuation over $X'$, then by \cite[Corollary 2.31(3)]{KM98} we have
		\begin{align*}
			a(F,X,B+M) &= a(F,X',B'+M') = a(F,X',B') \\ 
			&\geq a \big( F,X',(B')^+ \big) \geq 1+\varepsilon .
		\end{align*}
		This yields the statement.
	\end{proof}
	
	\begin{lem} \label{lem:blup}
		Let $ (X,B+M) $ be an NQC g-pair with data $X' \xto X\to Z$ and $ M'=\sum_{j=1}^l \mu_j M_j'$, where $B=\sum_{i=1}^s b_i B_i$ and the $B_i$ are distinct prime divisors, $\mu_j \ge 0$ and the $M_j'$ are $ \Q $-Cartier divisors which are nef over $Z$. Assume that $ X $ is smooth at the generic point of a codimension $k\ge2$ closed subvariety $V$ of $X$. Consider the blow-up of $ X $ along $ V $ and let $ E $ be the irreducible component of the exceptional divisor that dominates $ V $. Then
		$$ a(E,X,B+M)=k-\sum_{i=1}^s n_ib_i-\sum_{j=1}^l m_j\mu_j$$
		for some non-negative integers $n_1,\dots,n_s,m_1,\dots, m_l$.
	\end{lem}
	
	\begin{proof}
		We may assume that $E$ is a prime divisor on $X'$. By assumption and by \cite[Lemma 2.29]{KM98} we have $ a(E,X,B)=k-\sum_{i=1}^s n_ib_i$ for some non-negative integers $n_1,\dots,n_s$; note that $ n_i = \mult_V B_i $. Furthermore, by assumption and by the Negativity lemma \cite[Lemma 3.39(1)]{KM98}, we have
		$\mult_E (f^* M - M')=\sum_{j=1}^l m_j\mu_j$ for some non-negative integers $m_1,\dots,m_l$. Since
		$  a(E,X,B+M) =  a(E,X,B)-\mult_E (f^* M - M') $,
		the statement follows.
	\end{proof}
	
	We will also need the following lemma. Note that parts (i) and (ii) are analogs of \cite[Proposition 3.51 and Lemma 3.38]{KM98}, respectively, in the context of g-pairs, and their proofs are essentially identical as well, namely they follow readily from the Negativity lemma \cite[Lemma 3.39(1)]{KM98}.
	
	\begin{lem} \label{lem:monotonicity_discrep}
		Let $(X,B+M)$ and $(X',B'+M')$ be g-pairs such that there exists a diagram
		\begin{center}
			\begin{tikzcd}
				& W \arrow[dl, "g" swap] \arrow[dr, "g'"] \\
				X \arrow[rr, "\varphi", dashed] \arrow[dr, "f" swap] && X' \arrow[dl, "f'"] \\
				& Y ,
			\end{tikzcd}
		\end{center}
		where $Y$ and $W$ are normal varieties, all morphisms are proper birational, $ K_{X'} + B' + M' $ is $ f' $-nef, and there exists a nef $ \R $-Cartier $ \R $-divisor $ M_W $ on $ W $ with $ g_* M_W = M $ and $ g'_* M_W = M' $.
		\begin{enumerate}
			\item Assume that $B'=\varphi_*B+E$, where $E$ is the sum of all the $\varphi^{-1}$-exceptional prime divisors on $ X' $, and that
			$$a(F,X,B+M)\leq a(F,X',B'+M')$$ 
			for every $\varphi$-exceptional prime divisor $F$ on $X$. Then for any divisorial valuation $ F $ over $ X $ we have
			\[ a(F,X,B+M) \leq a(F,X',B'+M') . \]
			
			\item Assume that $ -(K_X + B + M) $ is $ f $-nef and that $ f_* B \geq f'_* B' $. Then for any divisorial valuation $ F $ over $ Y $ we have
			\[ a(F,X,B+M) \leq a(F,X',B'+M') , \]
			and the strict inequality holds if either 
			\begin{enumerate}
				\item $ -(K_X + B + M) $ is $ f $-ample and $ f $ is not an isomorphism above the generic point of $ c_Y(F)  $, or
				
				\item $ K_{X'} + B' + M' $ is $ f' $-ample and $ f' $ is not an isomorphism above the generic point of $ c_Y(F) $.
			\end{enumerate}
		\end{enumerate}
	\end{lem}
	
	\begin{rem}
		The operations of an MMP for a g-pair are analogous to those in the standard setting, see \cite[Section 4]{BZ16} and \cite[\S3.1]{HanLi} for the details as well as Subsection \ref{subsection:flips} for more information about flips in this context. Note that $\Q$-factorial and dlt singularities are preserved under these operations by \cite[Lemma 3.7]{HanLi}.
		
		Furthermore, by \cite[Remark 2.3 and Lemma 3.5]{HanLi} one can run an MMP with scaling of an ample divisor over $ Z $ for any lc g-pair $(X/Z,B+M)$ such that $ (X,0) $ is $ \Q $-factorial klt, and thus for any $\Q$-factorial dlt g-pair $(X/Z,B+M)$.
	\end{rem}
	
	The following result will be used frequently in the paper. Part (i) concerns the existence of a $\Q$-factorial terminalization of a klt g-pair and we give a detailed proof below. Part (ii) is \cite[Proposition 3.9]{HanLi} and we therefore omit its proof.
	
	\begin{lem}
		\label{lem:modifications}
		Let $(X,B+M)$ be a g-pair with data $ X' \overset{f}{\to} X \to Z $ and $ M' $. 
		\begin{enumerate}
			\item If $ (X,B+M) $ is klt, then, after possibly replacing $X'$ with a higher model, there exists a $\Q$-factorial terminal g-pair $(Y,\Delta+\Xi)$ with data $X' \overset{g}{\to} Y\to Z$ and $M'$, and a projective birational morphism $h \colon Y \to X$ such that 
			$ K_Y+\Delta+\Xi \sim_\R h^*(K_X+B+M) $. 
			The g-pair $(Y,\Delta+\Xi)$ is called a \emph{$ \Q $-factorial terminalization} of $(X,B +M)$.
			
			\item If $ (X,B+M) $ is lc, then, after possibly replacing $X'$ with a higher model, there exist a $\Q$-factorial dlt g-pair $(Y,\Delta+\Xi)$ with data $ X' \overset{g}{\to} Y \to Z $ and $ M' $, and a projective birational morphism $h \colon Y \to X$ such that 
			$K_Y + \Delta + \Xi \sim_\R h^*(K_X + B +M) $ and each $ h $-exceptional divisor is a component of $ \lfloor \Delta \rfloor $.
			The g-pair $(Y,\Delta+\Xi)$ is called a \emph{dlt blow-up} of $(X,B+M)$.
		\end{enumerate}
	\end{lem}
	
	\begin{proof}
		We show (i). By Proposition \ref{prop:klt_discrep} there are only finitely many exceptional divisorial valuations $E_1, \dots, E_k$ over $ X $ with $ a(E_j,X,B+M) \in (0,1]$ for any $ 1 \leq j \leq k $. We may assume that $f$ is a log resolution of $(X,B)$ such that each $ E_j $ is a prime divisor on $X'$. Thus, we may write
		$$ K_{X'}+B'+M' \sim_\R f^*(K_X+B+M) + F' , $$
		where $ B' $ is an effective $ \R $-divisor supported on $ f_*^{-1} B $ and the $ f $-exceptional divisors $ E_1, \dots, E_k $, and $F'$ is an effective $ f $-exceptional $ \R $-divisor that has no common components with $ B' $. In particular, $ (X',B'+M') $ is a $ \Q $-factorial klt g-pair such that $K_{X'}+B'+M'\equiv_X F'$. Now, by \cite[Lemma 4.4(2)]{BZ16} and by \cite[Proposition 3.8]{HanLi} we may run a $(K_{X'}+B'+M')$-MMP with scaling of an ample divisor over $ X $ which contracts only $F'$ and terminates with a model $h \colon Y \to X$ such that $K_Y+\Delta+\Xi \sim_\R h^*(K_X+B+M)$, where $\Delta$ is the strict transform of $B'$ on $ Y $ and $ \Xi $ is a pushforward of $M'$. Moreover, by construction the g-pair $(Y,\Delta+\Xi)$ is in fact $ \Q $-factorial and terminal, since the $ h $-exceptional divisors are the strict transforms of the $ E_j $ on $ Y $. This finishes the proof.
	\end{proof}

	\subsection{Minimal and log canonical models of generalized pairs}
	
	In this subsection we first recall the definitions of minimal models and log canonical models of g-pairs and then we briefly discuss some useful properties.
	
	\begin{dfn} \label{dfn:models}
		Let $(X,B +M)$ be an lc g-pair with data $ X' \overset{f}{\to} X \to Z $ and $ M' $. Let $ \varphi \colon X \dashrightarrow Y$ be a \emph{birational contraction} to a normal variety $ Y $ over $Z$, i.e., a birational map whose inverse does not contract any divisors. Assume that $ X' $ is a sufficiently high birational model of $ X $ so that the induced map $ g \colon X' \dashrightarrow Y $ is a morphism and, additionally, that $ (Y/Z,B_Y+M_Y) $ is a g-pair, where $ B_Y :=\varphi_*B$ and $ M_Y:=g_* M' $. Then $ (Y/Z,B_Y+M_Y) $ is called:
		\begin{enumerate}
			\item[(a)] a \emph{minimal model of $(X,B+M)$ over $ Z $} if
			\begin{itemize}
				\item $Y$ is not necessarily $\Q$-factorial if $X$ is not $\Q$-factorial, while $Y$ is (required to be) $\Q$-factorial if $X$ is $\Q$-factorial,
				
				\item $K_Y+B_Y+M_Y$ is nef over $Z$, and 
				
				\item we have 
				$$a(F,X,B+M) < a(F,Y,B_Y+M_Y)$$
				for any $\varphi$-exceptional prime divisor $ F $ on $ X $;
			\end{itemize}
			
			\item[(b)] a \emph{log canonical model of $(X,B+M)$ over  $Z$} if 
			\begin{itemize}
				\item $K_Y+B_Y+M_Y$ is ample over $Z$, and
				
				\item we have
				$$a(F,X,B+M) \leq a(F,Y,B_Y+M_Y)$$
				for any $\varphi$-exceptional prime divisor $ F $ on $ X $.
			\end{itemize}
		\end{enumerate}
	\end{dfn}
	
	With the same notation as in the above definition, note that if the g-pair $ (Y,B_Y+M_Y) $ is either a minimal model or a log canonical model of the lc g-pair $ (X,B+M) $ over $Z$, then $ (Y,B_Y+M_Y) $ is also lc by Lemma \ref{lem:monotonicity_discrep}(i).
	
	Moreover, the next result shows that any two minimal models of a given g-pair are isomorphic in codimension one, as in the case of usual pairs, see \cite[Theorem 3.52(2)]{KM98}. Its proof is analogous to the proof of \cite[Proposition 1.21]{Kol13}. Nevertheless, we provide all the details for the benefit of the reader.
	
	\begin{lem} \label{lem:MM_iso_codim1}
		Let $ \varphi_i \colon (X/Z,B+M) \dashrightarrow (X_i^m/Z,B_i^m+M_i^m) $, $ i \in \{1,2\} $, be two minimal models of an lc g-pair $ (X/Z,B+M) $. Then the induced map $ \varphi_2 \circ \varphi_1^{-1} \colon X_1^m \dashrightarrow X_2^m $ is an isomorphism in codimension one.
	\end{lem}
	
	\begin{proof}
		Let $ g \colon T \to X $ be a sufficiently high model so that the induced maps $ h_i \colon T \to X_i^m $ are actually morphisms and there exists an $ \R $-Cartier $ \R $-divisor $ M_T $ on $ T $ which is nef over $ Z $ and satisfies $ (h_i)_* M_T = M_i^m $, where $ i \in \{1,2\} $. Set
		$$ E_i := g^* (K_X + B +M) - h_i^* (K_{X_i^m} + B_i^m + M_i^m) , \ i \in \{1,2\} . $$
		By Lemma \ref{lem:monotonicity_discrep}(i) and by definition of a minimal model, $ E_i $ is an effective $ h_i $-exceptional $\R$-divisor and $ \Supp E_i $ contains the strict transforms of all the $ \varphi_i $-exceptional divisors. Subtracting the two formulas, we obtain
		$$ h_1^* (K_{X_1^m} + B_1^m + M_1^m) - h_2^* (K_{X_2^m} + B_2^m + M_2^m) = E_2 - E_1 . $$
		Since $ (h_1)_* (E_2 - E_1) \geq 0 $ and $ -(E_2 - E_1) $ is $ h_1 $-nef, by the Negativity lemma \cite[Lemma 3.39(1)]{KM98} we deduce that $ E_2 - E_1 \geq 0 $. Similarly, $ E_1 - E_2 \geq 0 $. Thus, $ E_1 = E_2 $, so $ \varphi_1 $ and $ \varphi_2 $ have the same exceptional divisors, whence the statement.
	\end{proof}	
	
	The following lemma illustrates how a minimal model and a log canonical model of a given g-pair are related to each other. It is an analog of \cite[Lemma 4.8.4]{Fuj17} in the context of g-pairs. Its proof is also similar, namely it follows readily from the Negativity and the Rigidity lemma; the details can be found in \cite{LMT}.
	
	\begin{lem} \label{lem:mapfromMMtoCM}
		Let $ (X/Z,B+M) $ be an lc g-pair. Let 
		$ (X^m,B^m+M^m) $ be a minimal model of $ (X,B+M) $ over $ Z $ and let $ (X^c,B^c+M^c) $ be a log canonical model of $ (X,B+M) $ over $ Z $. Then there exists a birational morphism 
		$ \alpha \colon X^m \to X^c $ such that
		\[ K_{X^m} + B^m + M^m \sim_\R \alpha^* (K_{X^c} + B^c + M^c) . \]
		In particular, $ K_{X^m} + B^m + M^m $ is semi-ample over $ Z $ and there exists a unique, up to isomorphism, log canonical model of $(X,B+M)$ over $ Z $. 
	\end{lem}

	\subsection{Flips and special termination for generalized pairs}
	\label{subsection:flips}
	
	First, we recall the definition of flips in the context of g-pairs.
	
	\begin{dfn} \label{dfn:flip}
		Let $(X/Z,B+M)$ be an lc g-pair. A \emph{$(K_X+B+M)$-flipping contraction} over $ Z $ is a projective birational morphism $ \theta \colon X \to W$ over $Z$ such that 
		\begin{itemize}
			\item $\codim_X \Exc(\theta) \geq 2$,
			
			\item $-(K_X+B+M)$ is ample over $ W $, and
			
			\item the relative Picard number is $\rho(X/W) = 1$.
		\end{itemize}
		A g-pair $(X^+/Z,B^++M^+)$ together with a projective birational morphism $\theta^+ \colon X^+ \to W$ over $Z$ is called a \emph{$(K_X+B+M)$-flip} of $\theta$ if
		\begin{itemize}
			\item $\codim_{X^+} \Exc(\theta^+) \geq 2$, and
			
			\item $K_{X^+}+ B^+ + M^+$ is ample over $W$, where $B^+$ is the strict transform of $B$ on $X^+$ and $ M^+ $ is defined as in Definition \ref{dfn:models}.
		\end{itemize}
		By a slight abuse of terminology, the induced birational map $\pi \colon X \dto X^+$ is also called a \emph{$(K_X+B+M)$-flip} over $ Z $. The map $ \theta^+ $ is called the \emph{flipped contraction} over $ Z $ and the set $ \Exc(\theta) $ (resp.\ $ \Exc(\theta^+) $) is called the \emph{flipping locus} (resp.\ the \emph{flipped locus}).
	\end{dfn}
	
	With the same notation as in the above definition, we emphasize that $ (X^+,B^+ + M^+) $ is the log canonical model of $ (X,B+M) $ over $ W $ by definition. Consequently, a flip is unique, if it exists, see Lemma \ref{lem:mapfromMMtoCM}. However, the existence of flips for g-pairs is not known yet in full generality; we refer to \cite[Section 4]{BZ16} and \cite[\S 3.1]{HanLi} for known special cases. 
	
	The next result, often referred to as the \emph{mononicity lemma}, describes the behaviour of log discrepancies under a flip and follows immediately from Lemma \ref{lem:monotonicity_discrep}(ii).
	
	\begin{lem} \label{lem:monotonicity}
		Let $ (X/Z,B+M) $ be an lc g-pair. If 
		$ \pi \colon (X,B+M) \dashrightarrow (X^+,B^+ + M^+) $
		is a $ (K_X+B+M) $-flip over $Z$, then for any divisorial valuation $E$ over $X$ we have 
		$$ a(E,X,B+M)\le  a(E,X^+,B^++M^+) , $$
		and the strict inequality holds if and only if either $\Center_{X}E$ is contained in the flipping locus or $\Center_{X^+}E$ is contained in the flipped locus.
	\end{lem}
	
	In the remainder of this subsection we will be primarily concerned with sequences of flips. First, we recall the following:
	Let
	\begin{center}
		\begin{tikzcd}[column sep = 1.25em, row sep = 1.75em]
			(X_1,B_1+M_1) \arrow[dr, "\theta_1" swap] \arrow[rr, dashed, "\pi_1"] && (X_2,B_2+M_2) \arrow[dl, "\theta_1^+"] \arrow[dr, "\theta_2" swap] \arrow[rr, dashed, "\pi_2"] && (X_3,B_3+M_3) \arrow[dl, "\theta_2^+"] \arrow[rr, dashed, "\pi_3"] && \dots \\
			& Z_1 && Z_2
		\end{tikzcd}
	\end{center}
	be a sequence of flips over $Z$ starting from an lc g-pair $(X_1/Z,B_1+M_1)$. \emph{Special termination}\index{Special termination} of the given sequence of flips means that after finitely many steps the flipping locus $ \Exc(\theta_i) $ does not intersect the non-klt locus of $ (X_i/Z,B_i+M_i) $, i.e., there exists an integer $ k \geq 1$ such that
	$$ \Exc(\theta_i)\cap\nklt(X_i,B_i+M_i)=\emptyset \text{ for all } i \geq k . $$
	Moreover, the phrase \emph{special termination for (NQC) lc g-pairs of dimension $ d $}, which will be encountered in Section \ref{section:term_for_psef}, means that the above applies to any sequence of flips starting from any (NQC) lc g-pair of dimension $ d $.
	
	The next result constitutes one of the main ingredients for the proof of Theorem \ref{thm:gklt}. It allows us to pass from a sequence of flips for klt g-pairs to a sequence of flips for terminal g-pairs under certain conditions which can be achieved with serious effort in the setting of Theorem \ref{thm:gklt}.
	
	\begin{lem} \label{lem:lift_klt_term}
		Let $ (X_1/Z,B_1+M_1) $ be a $ \Q $-factorial klt g-pair and let 
		\begin{center}
			\begin{tikzcd}[column sep = 0.8em, row sep = 1.75em]
				(X_1,B_1+M_1) \arrow[dr] \arrow[rr, dashed, "\pi_1"] && (X_2,B_2+M_2) \arrow[dl] \arrow[dr] \arrow[rr, dashed, "\pi_2"] && (X_3,B_3+M_3) \arrow[dl] \arrow[rr, dashed, "\pi_3"] && \dots \\
				& Z_1 && Z_2
			\end{tikzcd}
		\end{center}
		be a sequence of flips over $ Z $ starting from $ (X_1/Z,B_1+M_1) $. Assume that for each $ i \geq 1 $ there exists a $ \Q $-factorial terminalization $ (Y_i', \Delta_i' + \Xi_i') $ of $ (X_i, B_i + M_i) $, such that each $Y_{i+1}'$ is isomorphic in codimension one to $ Y_i' $ and each $\Delta_{i+1}'$ is the strict transform of $\Delta_i'$. Then there exists a diagram
		\begin{center}
			\begin{tikzcd}[column sep = 0.8em, row sep = large]
				(Y_1,\Delta_1+\Xi_1) \arrow[d, "h_1" swap] \arrow[rr, dashed, "\rho_1"] && (Y_2,\Delta_2+\Xi_2) \arrow[d, "h_2" swap] \arrow[rr, dashed, "\rho_2"] && (Y_3,\Delta_3+\Xi_3) \arrow[d, "h_3" swap] \arrow[rr, dashed, "\rho_3"] && \dots \\ 
				(X_1,B_1+M_1) \arrow[dr] \arrow[rr, dashed, "\pi_1"] && (X_2,B_2+M_2) \arrow[dl] \arrow[dr] \arrow[rr, dashed, "\pi_2"] && (X_3,B_3+M_3) \arrow[dl] \arrow[rr, dashed, "\pi_3"] && \dots \\
				& Z_1 && Z_2
			\end{tikzcd}
		\end{center}
		where, for each $ i \geq 1 $, the map $ \rho_i $ is a sequence of $ (K_{Y_i} + \Delta_i + \Xi_i) $-flips over $ Z_i $ and the map $ h_i $ is a $ \Q $-factorial terminalization of $ (X_i,B_i+M_i) $. In particular, we obtain overall a sequence of flips over $ Z $ starting from the $ \Q $-factorial terminal g-pair $ (Y_1, \Delta_1 + \Xi_1) $.
	\end{lem}
	
	\begin{proof}
		Fix $ i \geq 1 $. Since $ (X_{i+1},B_{i+1}+M_{i+1}) $ is the log canonical model of $ (X_i,B_i+M_i) $ over $ Z_i $, by assumption and by construction of a $ \Q $-factorial terminalization it follows that $ (Y_{i+1}',\Delta_{i+1}'+\Xi_{i+1}') $ is a minimal model of $ (Y_i',\Delta_i' + \Xi_i') $ over $ Z_i $.
		
		Set $ (Y_1, \Delta_1 + \Xi_1) := (Y_1',\Delta_1' + \Xi_1') $ and denote by $ h_1 \colon (Y_1, \Delta_1 + \Xi_1) \to (X_1, B_1 + M_1) $ the corresponding morphism. By \cite[Lemma 4.4(2)]{BZ16} we may run some $ (K_{Y_1} + \Delta_1 + \Xi_1) $-MMP with scaling of an ample divisor over $ Z_1 $ which terminates with a minimal model $ (Y_2, \Delta_2 + \Xi_2) $ of $ (Y_1, \Delta_1 + \Xi_1) $ over $ Z_1 $. Since $ (X_2,B_2+M_2) $ is the log canonical model of $ (Y_1,\Delta_1+\Xi_1) $ over $ Z_1 $, it follows from Lemma \ref{lem:mapfromMMtoCM} that there exists a projective birational morphism $ h_2 \colon Y_2 \to X_2 $ such that $ K_{Y_2} + \Delta_2 + \Xi_2 \sim_\R h_2^* (K_{X_2} + B_2 + M_2) $. Since $ (Y_2',\Delta_2'+\Xi_2') $ is a minimal model of $ (Y_1,\Delta_1 + \Xi_1) $ over $ Z_1 $, by Lemma \ref{lem:MM_iso_codim1} we deduce that this MMP with scaling over $ Z_1 $ consists only of flips, and therefore $ h_2 \colon (Y_2, \Delta_2 + \Xi_2) \to (X_2,B_2+M_2) $ is a $ \Q $-factorial terminalization of $ (X_2,B_2+M_2) $. Thus, we obtain the required diagram by repeating this procedure.
	\end{proof}
	
	The following result allows us to pass from a sequence of flips for lc g-pairs to a sequence of flips for dlt g-pairs under some mild assumptions in lower dimensions. Its proof is similar to the proof of Lemma \ref{lem:lift_klt_term}. One of the basic differences, though, is that \cite[Theorem 4.4(ii)]{LT19} must be invoked instead of \cite[Lemma 4.4(2)]{BZ16}. This is actually the reason why we have to make those assumptions in lower dimensions. For the details of the proof of Lemma \ref{lem:lift_lc_dlt} we refer to \cite{LMT}, where a slightly more general version of this result was established. Incidentally, we emphasize that \cite{LMT} does not depend on \cite{Mor18}.
	
	\begin{lem}\label{lem:lift_lc_dlt}
		Assume the existence of minimal models for smooth varieties of dimension $ n-1 $.
		
		Let $ (X_1/Z,B_1+M_1) $ be an NQC lc g-pair  of dimension $ n $ and let
		\begin{center}
			\begin{tikzcd}[column sep = 0.8em, row sep = 1.75em]
				(X_1,B_1+M_1) \arrow[dr] \arrow[rr, dashed, "\pi_1"] && (X_2,B_2+M_2) \arrow[dl] \arrow[dr] \arrow[rr, dashed, "\pi_2"] && (X_3,B_3+M_3) \arrow[dl] \arrow[rr, dashed, "\pi_3"] && \cdots \\
				& Z_1 && Z_2
			\end{tikzcd}
		\end{center}
		be a sequence of flips over $Z$ starting from $(X_1/Z,B_1+M_1)$. Then there exists a diagram
		\begin{center}
			\begin{tikzcd}[column sep = 0.8em, row sep = large]
				(Y_1,\Delta_1+\Xi_1) \arrow[d, "h_1" swap] \arrow[rr, dashed, "\rho_1"] && (Y_2,\Delta_2+\Xi_2) \arrow[d, "h_2" swap] \arrow[rr, dashed, "\rho_2"] && (Y_3,\Delta_3+\Xi_3) \arrow[d, "h_3" swap] \arrow[rr, dashed, "\rho_3"] && \cdots 
				\\ 
				(X_1,B_1+M_1) \arrow[dr] \arrow[rr, dashed, "\pi_1"] && (X_2,B_2+M_2) \arrow[dl] \arrow[dr] \arrow[rr, dashed, "\pi_2"] && (X_3,B_3+M_3) \arrow[dl] \arrow[rr, dashed, "\pi_3"] && \cdots \\
				& Z_1 && Z_2
			\end{tikzcd}
		\end{center}
		where, for each $i \geq 1$, the map $ \rho_i \colon Y_i \dashrightarrow Y_{i+1} $ is a $(K_{Y_i}+\Delta_i+\Xi_i)$-MMP over $ Z_i $ and the map $ h_i $ is a dlt blow-up of $ (X_i,B_i+M_i) $. In particular, we obtain overall an MMP over $ Z $ for the NQC $\Q$-factorial dlt g-pair $ (Y_1,\Delta_1+\Xi_1) $.
	\end{lem}
	
	The next result plays a fundamental role in the proof of Theorem \ref{thm:mainthm}. We use here (as well as in Theorem \ref{thm:mainthm}) the definition of a flip with respect to an arbitrary $\R$-Cartier $\R$-divisor, see \cite[Definition 2.3]{Bir07}.
	
	\begin{lem} \label{lem:supplement_g}
		Let $ (X,B+M) $ be an lc g-pair with data $ X' \overset{f}{\to} X \to Z $ and $ M' $. Let
		\begin{center}
			\begin{tikzcd}
				(X,B+M) \arrow[dr, "\theta" swap] \arrow[rr, dashed, "\pi"] && (X^+,B^+ + M^+) \arrow[dl, "\theta^+"] \\
				& W
			\end{tikzcd}
		\end{center}
		be a flip over $ Z $ such that the flipping locus $ \Exc(\theta) $ does not intersect the non-klt locus of the g-pair $ (X,B+M) $, that is, $ \Exc(\theta) \cap \nklt(X,B+M) = \emptyset $. Let $ h \colon (Y,\Delta+\Xi) \to (X,B+M) $ be a dlt blow-up of $ (X,B+M) $, where the g-pair $(Y,\Delta+\Xi)$ comes with data $ X' \overset{g}{\to} Y \to Z $ and $ M' $;
		we may assume that $ X' $ is a sufficiently high model so that $ f = h \circ g $. Run a $ (K_Y+\Delta+\Xi) $-MMP with scaling of an ample divisor over $ W $ and let 
		\begin{center}
			\begin{tikzcd}
				(Y_i,\Delta_i + \Xi_i) \arrow[dr, "\theta_i" swap] \arrow[rr, dashed, "\pi_i"] && (Y_{i+1}, \Delta_{i+1} + \Xi_{i+1}) \arrow[dl, "\theta_i^+"] \\
				& W_i 
			\end{tikzcd}
		\end{center}
		be its intermediate steps, where $ (Y_1,\Delta_1 + \Xi_1) := (Y,\Delta+\Xi) $. Then:
		\begin{enumerate}
			\item This $ (K_Y+\Delta+\Xi) $-MMP over $ W $ consists only of flips, at each step the flipping locus $ \Exc(\theta_i) $ avoids the non-klt locus of $ (Y_i,\Delta_i + \Xi_i) $, that is, $ \Exc(\theta_i) \cap \Supp \Delta_i^{=1} = \emptyset $ for every $ i \geq 1 $, and thus the sequence of $ (K_Y+\Delta+\Xi) $-flips over $ W $ is also a sequence of $ \big(K_Y+ \Delta^{<1} + \Xi \big) $-flips over $ W $, where $ \big(Y,\Delta^{<1} +\Xi \big) $ is a $ \Q $-factorial klt g-pair.
			
			\item Assume now that there exist an $ \R $-divisor $ P' $ on $ X' $ which is nef over $ Z $ and an effective $ \R $-divisor $ N $ on $ X $ such that $ P+N $ is $ \R $-Cartier, $ -(P+N) $ is ample over $ W $ and $ h^*(P+N) = Q + \Lambda + F $, where $ P := f_* P' $, $ Q := g_* P' $, $ \Lambda \geq 0 $ and $ \Supp F \subseteq \Supp \lfloor \Delta \rfloor $. Then the sequence of $ (K_Y+\Delta+\Xi) $-flips over $ W $ is also a sequence of $ (Q+\Lambda) $-flips over $ W $.
		\end{enumerate}
	\end{lem}
	
	\begin{proof}~
		\noindent (i) First, we show that
		\begin{equation} \label{eq:1_suppl_g}
			\Exc(\theta_i) \cap \Supp \lfloor \Delta_i \rfloor = \emptyset \ \text{ for every } i \geq 1 . 
		\end{equation}
		To this end, set $ U := X \setminus \Exc(\theta) $, $ V_1 := h^{-1}(U) $, $ T = \theta(U) $, and note that $ T $ is an open subset of $ W $. Since $ \pi |_U \colon U \to \pi(U) $ is an isomorphism, $ (K_X + B + M)|_U  $ is trivially semi-ample over $ T $, and thus
		\[ (K_{Y_1} + \Delta_1 + \Xi_1)|_{V_1} = \left( h|_{V_1} \right)^*\big( (K_X+B+M)|_U \big) \]
		is also semi-ample over $ T $. Therefore, $ V_1 \cap \Exc(\theta_1) = \emptyset $, so $ V_1 $ is contained in the locus where $ \pi_1 $ is an isomorphism, and we may thus consider its isomorphic image $ V_2 := \pi_1(V_1) $ and infer that $ (K_{Y_2} + \Delta_2 + \Xi_2)|_{V_2} $ is semi-ample over $ T $, hence $ V_2 \cap \Exc(\theta_2) = \emptyset $. By proceeding analogously, if we set $ V_i := \pi_{i-1} (V_{i-1}) $, then we obtain that
		\begin{equation} \label{eq:2_suppl_g}
			V_i \cap \Exc(\theta_i) = \emptyset \ \text{ for all } i \geq 1 .
		\end{equation}
		We can now readily derive \eqref{eq:1_suppl_g} from \eqref{eq:2_suppl_g} as follows. We first note that
		\[ V_1 \supseteq \nklt(Y_1,\Delta_1+\Xi_1) = \Supp \lfloor \Delta_1 \rfloor , \]
		since $ \nklt(X,B+M) \subseteq U $ by assumption, and subsequently, by construction and by the above relation, we deduce inductively that 
		\[ V_i \supseteq \nklt(Y_i,\Delta_i+\Xi_i) = \Supp \lfloor \Delta_i \rfloor \ \text{ for all } i \geq 2 . \]
		Consequently, we obtain \eqref{eq:1_suppl_g} by combining the above two relations with \eqref{eq:2_suppl_g}.
		
		Next, to prove that the $ (K_{Y_1}+\Delta_1+\Xi_1) $-MMP with scaling over $ W $ consists only of flips, we argue by contradiction. If a divisorial contraction $ \theta_i \colon Y_i \to W_i $ with exceptional prime divisor $ E_i := \Exc(\theta_i) $ appears at the $ i $-th step of this MMP (we may assume that $ i $ is the smallest such index), then by \eqref{eq:1_suppl_g} we have $ E_i \cap \Supp \lfloor \Delta_i \rfloor = \emptyset $, so the strict transform $ E_1 $ of $ E_i $ on $ Y_1 $ cannot be a component of $ \lfloor \Delta_1 \rfloor $. On the other hand, $ E_1 $ must be contracted over $W$, and since $ \theta $ is a small map, $ E_1 $ is an $ h $-exceptional divisor, and thus a component of $ \lfloor \Delta_1 \rfloor $ by construction of a dlt blow-up, a contradiction.
		
		Finally, if we denote by $ R_i $ the $ (K_{Y_i}+\Delta_i+\Xi_i) $-negative extremal ray contracted at the $ i $-th step of the given sequence of flips, then \eqref{eq:1_suppl_g} implies
		$$ \big( K_{Y_i} + \Delta_i^{<1} + \Delta_i^{=1} + \Xi_i \big) \cdot R_i = \big( K_{Y_i} + \Delta_i^{<1} + \Xi_i \big) \cdot R_i , $$
		whence the last assertion of (i).
		
		\medskip
		
		\noindent (ii) For every $ i \geq 1 $ we denote by $ Q_i $, $ \Lambda_i $ and $ F_i $ the strict transforms on $ Y_i $ of $ Q_1 := Q $, $ \Lambda_1 := \Lambda $ and $ F_1 := F $, respectively. Since $ -(P+N) $ is ample over $ W $ and since $ \rho(X/W) = 1 $, we have $ K_X + B + M \equiv_W \alpha (P + N) $ for some $ \alpha > 0 $, and thus
		\[ K_Y + \Delta + \Xi = h^*(K_X + B + M) \equiv_W \alpha h^* (P + N) = \alpha (Q + \Lambda + F) . \]
		Therefore, 
		\begin{equation} \label{eq:3_suppl_g}
			K_{Y_i} + \Delta_i + \Xi_i \equiv_W \alpha (Q_i + \Lambda_i + F_i) \ \text{ for every } i \geq 1 .
		\end{equation}
		Moreover, by assumption and by (i) we infer that
		\begin{equation} \label{eq:4_suppl_g}
			\Supp F_i \subseteq \Supp \lfloor \Delta_i \rfloor \ \text{ for every } i \geq 1 .
		\end{equation}
		Hence, if $ C_i \subseteq Y_i $ is a flipping curve at the $ i $-th step of the sequence of $ (K_{Y_1} + \Delta_1 + \Xi_1) $-flips over $ W $, then by \eqref{eq:1_suppl_g}, \eqref{eq:3_suppl_g} and \eqref{eq:4_suppl_g} we obtain
		\[ (Q_i+\Lambda_i) \cdot C_i = (Q_i + \Lambda_i + F_i) \cdot C_i = \frac{1}{\alpha} (K_{Y_i} + \Delta_i + \Xi_i) \cdot C_i < 0 . \]
		In other words, $ (Q_i + \Lambda_i) \cdot R_i < 0 $ for every $ i \geq 1 $, as desired.
	\end{proof}

	\subsection{The difficulty}
	
	The notion of the \emph{difficulty} of a terminal variety was originally introduced by Shokurov in \cite{Sho85} in order to prove the termination of flips for terminal 3-folds. Since then several variants of this notion have appeared in the literature, see \cite{Kaw92,Kol92,Fuj04,Fuj05,AHK07}, as the main idea to tackle the termination of flips conjecture has remained the same. Here, we introduce a version of the difficulty in the setting of NQC g-pairs, which generalizes \cite[Definition 6.20]{KM98} and plays a key role in the proof of Theorem \ref{thm:gterminal}.
	
	\begin{dfn} \label{dfn:diff}
		Let $(X,B+M)$ be a $ \Q $-factorial NQC terminal g-pair with data $X'\xto X\to Z$ and $M'=\sum_{j=1}^l\mu_j M_j'$, where $B=\sum_{i=1}^s b_i B_i$, $ b_i \in [0,1) $ and the $B_i$ are distinct prime divisors, $\mu_j\ge0$ and the $M_j'$ are $ \Q $-Cartier divisors which are nef over $Z$. Set $b:=\max\{b_1, \dots, b_s\}$ ($ b:=0 $ if $ B=0 $), consider the set
		$$ \mathcal{S}_b :=\bigg\{ \sum_{i=1}^sn_ib_i+\sum_{j=1}^lm_j\mu_j \ \Big| \ n_i,m_j\in \Z_{\ge0} \bigg\}\cap[b,+\infty) $$
		($ S_0 = \{0\}$ if $ B=M=0 $), and for every $ \xi \in \mathcal{S}_b $, set
		\begin{align*}
			d_{\xi}(X,B+M) := \# \big\{ E \mid E & \text{ is an exceptional divisorial valuation} \\ 
			&\text{ over $X$ with }  a(E,X,B+M)<2-\xi \big\}.
		\end{align*}
		The \emph{difficulty function} of $(X,B+M)$ is defined as
		$$ d(X,B+M):= \sum_{\xi\in\mathcal{S}_b} d_{\xi}(X,B+M) . $$
	\end{dfn}
	
	We show that the difficulty is finite and non-increasing under flips.
	
	\begin{lem} \label{lem:difficulty}
		Assume the notation from Definition \ref{dfn:diff}. Then
		$$ 0 \leq d(X,B+M) < +\infty . $$
		Moreover, if
		$ \pi \colon (X,B+M) \dto (X^+,B^+ + M^+) $
		is a $ (K_X+B+M) $-flip over $ Z $, then
		$$ d(X,B+M) \geq d(X^+,B^++M^+) . $$
	\end{lem}
	
	\begin{proof}
		First, we deal with the finiteness of the difficulty. For every $ \xi \in \mathcal{S}_b \cap [1, + \infty) $ we clearly have $ d_{\xi}(X,B+M) = 0 $ as $(X,B+M)$ is terminal. On the other hand, there are only finitely many $ \xi \in \mathcal{S}_b \cap [0,1) $ and for each such $ \xi $ it holds that $ d_{\xi}(X,B+M) < + \infty $ by Proposition \ref{prop:klt_discrep}, since we necessarily have $ 2 - \xi \leq 1 + \varepsilon = 2 - b $. Hence, $ d(X,B+M) < +\infty $. Additionally, the difficulty is obviously non-negative, while Lemma \ref{lem:monotonicity} shows that it is non-increasing under a flip. This completes the proof.
	\end{proof}

	\section{Termination of Flips in Dimension 3} 
	\label{section:term_in_dim_3}
	
	In this section we prove Theorem \ref{thm:terminationforthreefolds}. The proof consists essentially of three steps. First, we show the termination of flips for $3$-dimensional $\Q$-factorial NQC terminal g-pairs by invoking Lemma \ref{lem:difficulty} (see Theorem \ref{thm:gterminal}). Our proof is inspired by the previous results on the termination of flips for $3$-dimensional terminal pairs. Second, we establish the termination of flips for $3$-dimensional $\Q$-factorial NQC klt g-pairs (see Theorem \ref{thm:gklt}). Specifically, we prove analogs of \cite[Proposition 4.4 and Lemma 4.4.1]{Sho96} in the context of g-pairs (see Propositions \ref{prop:bddci} and \ref{prop:finitelogdiscrepancy}), which, together with Lemma \ref{lem:lift_klt_term}, enable us to derive Theorem \ref{thm:gklt} from Theorem \ref{thm:gterminal}. Here, we follow the same line of arguments as in \cite{Kaw92,Sho96}. Third, we obtain Theorem \ref{thm:terminationforthreefolds} by combining Lemma \ref{lem:lift_lc_dlt} and the special termination for NQC $\Q$-factorial dlt g-pairs of dimension $3$ (see \cite[Theorem 5.1]{LMT} and note that it does not depend on \cite{Mor18}) with Theorem \ref{thm:gklt}.
	
	\begin{thm}\label{thm:gterminal}
		Let $(X_1/Z,B_1+M_1)$ be a $3$-dimensional $\Q$-factorial NQC terminal g-pair. Then any sequence of flips over $ Z $ starting from $(X_1/Z,B_1+M_1)$ terminates.
	\end{thm}
	
	\begin{proof}
		Let 
		\begin{center}
			\begin{tikzcd}[column sep = 0.8em, row sep = 1.75em]
				(X_1,B_1+M_1) \arrow[dr, "\theta_1" swap] \arrow[rr, dashed, "\pi_1"] && (X_2,B_2+M_2) \arrow[dl, "\theta_1^+"] \arrow[dr, "\theta_2" swap] \arrow[rr, dashed, "\pi_2"] && (X_3,B_3+M_3) \arrow[dl, "\theta_2^+"] \arrow[rr, dashed, "\pi_3"] && \dots \\
				& Z_1 && Z_2
			\end{tikzcd}
		\end{center}
		be a sequence of flips over $Z$ starting from the g-pair $(X_1,B_1+M_1)$ which comes with data $X_1' \overset{f_1}{\longrightarrow} X_1 \to Z$ and $M_1'=\sum_{j=1}^l \mu_j M_{1,j}'$, where $\mu_j\ge0$ and the $M_{1,j}'$ are $ \Q $-Cartier divisors which are nef over $Z$. Note that each g-pair $ (X_i,B_i+M_i) $ in this sequence of flips is terminal by Lemma \ref{lem:monotonicity} and comes with data $ X_1' \overset{f_i}{\longrightarrow} X_i \to Z $ and $ M_1' $. We will prove the statement by induction on the number of components of $B_1$. 
		
		We first consider the case when $B_1=0$. Since $ X_{i+1} $ is terminal, it is smooth at the generic point of every flipped curve. If $ C_{i+1} $ is a flipped curve with the generic point $ \eta_{i+1} \in X_{i+1} $ and if $E_{i+1}$ is the prime divisor obtained by blowing up $ X_{i+1} $ at $ \eta_{i+1} $, then by Lemmata \ref{lem:blup} and \ref{lem:monotonicity} we have 
		$$ a(E_{i+1},X_i,M_{i}) <  a(E_{i+1},X_{i+1},M_{i+1})=2-\sum_j m_j\mu_j $$
		for some non-negative integers $m_j$, and therefore $d(X_i,M_{i})>d(X_{i+1},M_{i+1})$. Thus, if the given sequence of flips were infinite, then we would obtain a strictly decreasing sequence $ \left\{ d(X_i,M_i) \right\}_{i=1}^{+\infty} $ of non-negative integers with $ d(X_1,M_1) < +\infty $, which is impossible. Hence the given sequence of flips over $ Z $ terminates.
		
		Assume now that the statement holds when the number of components of $B_1$ is $\leq s-1$, and suppose that the number of components of $ B_1$ is $s \geq 1$. Write $B_1 = \sum_{k=1}^s b_k B_{1,k}$, where $ b_k \in (0,1) $ and the $B_{1,k}$ are distinct prime divisors. Set $ b:= \max\{b_1, \dots, b_s\} $, let $D_1:=\sum_{\{k \mid  b_k=b\}}B_{1,k}$, and denote by $B_{i,k}$ (resp.\ $D_i$) the strict transforms of $B_{1,k}$ (resp.\ $ D_1 $) on $X_i$ for any $i,k$.
		
		\begin{claim}\label{claim:nocontain}
			After finitely many flips, no flipping curve 
			is contained in the strict transforms $ D_i $ of $D_1$.
		\end{claim}
		
		\noindent Grant the above claim for a moment and relabel the given sequence of flips so that no flipping curve is contained in the strict transforms of $ D_1 $. Fix $ i \geq 1 $. Then $ D_i \cdot C \geq 0 $ for all flipping curves $ C \subseteq X_i $, and thus the $ (K_{X_i} + B_i + M_i) $-flip over $ Z $ is also a $ \big( K_{X_i} + (B_i - bD_i) + M_i \big) $-flip over $ Z $. Consequently, the given sequence of $ (K_{X_i} + B_i + M_i) $-flips over $Z$ is also a sequence of $ \big( K_{X_i} + (B_i - bD_i) + M_i \big) $-flips over $ Z $, where $ B_i - bD_i = \sum_{\{ k \mid b_k < b \}} b_k B_{i,k} $, and hence it terminates by induction. It remains to prove Claim \ref{claim:nocontain}.
		
		\begin{proof}[Proof of Claim \ref{claim:nocontain}]
			We first show that eventually no flipped curve is contained in the strict transforms $ D_i $ of $D_1$. Note that whenever a flipped curve $ C_{i+1} \subseteq X_{i+1} $ is contained in $ \Supp D_{i+1} $, then it holds that 
			$$ d(X_i,B_i+M_i) > d(X_{i+1},B_{i+1}+M_{i+1}) . $$
			Indeed, using notation and arguing as in the second paragraph of the proof, we have
			\begin{align*}
				a(E_{i+1},X_i,B_i+M_i) &<  a(E_{i+1},X_{i+1},B_{i+1}+M_{i+1}) \\
				&= 2 - \sum_k n_k b_k - \sum_j m_j\mu_j
			\end{align*}
			for some non-negative integers $ n_k = \mult_{C_{i+1}} B_{i+1,k} $ and $m_j$ such that
			\[ \xi := \sum_k n_k b_k + \sum_j m_j\mu_j \geq \sum_{\{k\mid b_k=b\}} n_k b_k \geq b  , \]
			which yields the assertion. Since $ d(X_1,B_1+M_1) < + \infty $ and since the difficulty takes values in $ \Z_{\geq 0} $, this situation can occur only finitely many times in the given sequence of flips over $ Z $. In other words, after finitely many flips, no flipped curve is contained in the strict transforms $ D_i $ of $D_1$; by relabelling the given sequence of flips, we may assume that this holds for all $ i \geq 1 $.
			
			We prove next that eventually no flipping curve is contained in the strict transforms $ D_i $ of $D_1$. To this end, fix $ i \geq 1 $ and consider the corresponding step
			\begin{center}
				\begin{tikzcd}[column sep = 2em]
					(X_i,B_i+M_i) \arrow[dr, "\theta_i" swap] \arrow[rr, "\pi_i", dashed] 
					&& (X_{i+1},B_{i+1}+M_{i+1}) \arrow[dl, "\theta_i^+"] \\
					& Z_i 
				\end{tikzcd}
			\end{center}
			of the given sequence of flips over $ Z $. We will construct below the following commutative diagram:
			\begin{center}
				\begin{tikzcd}[column sep = 1.75em]
					D_i^\nu \arrow[d, "\nu_i" swap] \arrow[ddrr, "\varphi_i" swap, shorten >= 1.5em] \arrow[rrrr, "g_i \, = \, \psi_{i+1}^{-1} \, \circ \ \varphi_i"] & & & & D_{i+1}^\nu \arrow[d, "\nu_{i+1}"] \arrow[ddll, "\psi_{i+1}", shorten >= 1.5em] \\ [0.5em]
					D_i \arrow[ddrr, "\theta_i |_{D_i}" swap, bend right = 25pt] & & & & D_{i+1} \arrow[ddll, "\theta_i^+ |_{D_{i+1}}", bend left = 25pt] \\
					&& \left( \theta_i (D_i) \right)^\nu = \left( \theta_i^+ (D_{i+1}) \right)^\nu \arrow[d] \\ [0.5em]
					&& \theta_i (D_i) = \theta_i^+ (D_{i+1})
				\end{tikzcd}
			\end{center}
			Let $ \nu_i \colon D_i^\nu \to D_i $ be the normalization of $D_i$ and let $ \varphi_i \colon D_i^\nu \to \left( \theta_i (D_i) \right)^\nu $ be the morphism between the normalizations induced by $ \nu_i \circ \theta_i |_{D_i} $. Since $ D_{i+1} $ contains no flipped curves, the restriction $ \theta_i^+ |_{D_{i+1}} \colon D_{i+1} \to \theta_i^+ (D_{i+1}) $ of $ \theta_i^+ $ to $ D_{i+1} $ is an isomorphism away from finitely many points, so it is a finite birational morphism. Thus, the composite $ \nu_{i+1} \circ \theta_i^+ |_{D_{i+1}} $ is a finite birational morphism and it induces a morphism $ \psi_{i+1} \colon D_{i+1}^\nu \to \left( \theta_i^+ (D_{i+1}) \right)^\nu $. Clearly, the map $ \psi_{i+1} $ is birational, and it is also finite by construction, as it contracts no curves. Therefore, $ \psi_{i+1} $ is an isomorphism, since $ \left( \theta_i^+ (D_{i+1}) \right)^\nu $ is normal. Observe now that $ \theta_i (D_i) = \theta_i^+ (D_{i+1}) $, since $ \pi_i $ is an isomorphism in codimension one. Hence, we obtain a projective birational morphism $ g_i := \psi_{i+1}^{-1} \circ \varphi_i \colon D_i^\nu \to D_{i+1}^\nu $, which, by construction, can only contract curves $ \gamma \subseteq D_i^\nu $ whose image $ \nu_i(\gamma) \subseteq D_i $ is a flipping curve contained in $ \Supp D_i $.
			
			Consider the minimal resolution $ q_1 \colon D_1^\mu \to D_1^\nu $ of $ D_1^\nu $ and for each $ i \geq 1 $ the induced map $ \zeta_i := g_{i-1} \circ \dots \circ g_1 \circ q_1 \colon D_1^\mu \to D_i^\nu $, which is in particular a resolution of $ D_i^\nu $. By \cite[Lemma 1.6(2)]{AHK07} we obtain a non-increasing sequence of positive integers $ \{ \rho_i \}_{i=1}^{+\infty} $, where
			$$ \rho_i := \rho \big( D_1^\mu/Z \big) - \# \{ \text{exceptional prime divisors of } \zeta_i \} , \ i \geq 1 . $$
			Since $ \rho_1 < + \infty $, the sequence $ \{ \rho_i \}_{i=1}^{+\infty} $ stabilizes. By construction of $ g_i $, this implies that $D_i$ contains no flipping curves for sufficiently large $ i $. In other words, after finitely many flips no flipping curve is contained in the strict transforms of $D_1$, as claimed.
		\end{proof}
		
		\noindent This completes the proof.
	\end{proof}
	
	We thank Jingjun Han for showing us the following result. We provide a detailed proof for the benefit of the reader.
	
	\begin{lem}\label{lem:length}
		Let $X$ be a normal quasi-projective variety of dimension $3$ and let $S$ be a prime divisor on $X$ such that $K_X+S$ is $\Q$-Cartier. If there exists a contraction $h \colon S \to T$ to a variety $T$ with $\dim T \le1$, then there exists a curve $C \subseteq S$ which is contracted by $ h $ and such that $(K_X+S) \cdot C\geq -3.$
	\end{lem}
	
	\begin{proof}
		Let $ \nu \colon S^{\nu} \to S $ be the normalization of $S$ and let $f \colon S' \to S^{\nu}$ be the minimal resolution of $S^{\nu}$. By adjunction there exists an effective $\Q$-divisor $\operatorname{Diff}_{S^\nu}(0)$ on $S^\nu$ such that $ (K_X+S) |_{S^\nu} \sim_\Q K_{S^{\nu}}+\operatorname{Diff}_{S^\nu}(0)$ (see \cite[Section 4.1]{Kol13}). Consider the divisor $ f^* \big( K_{S^{\nu}}+\operatorname{Diff}_{S^\nu}(0) \big) $, where the \emph{pullback} is taken \emph{in the sense of Mumford} (see \cite[Remark 4-6-3(i)]{Mat02}). Since this pullback is linear and respects effectivity (see \cite[Section II(b)]{Mum61}), we have $ f^* \big( K_{S^{\nu}}+\operatorname{Diff}_{S^\nu}(0) \big) \geq f^*(K_{S^{\nu}}) $. Moreover, by \cite[Theorem 4-6-2]{Mat02} we deduce that $ f^*(K_{S^{\nu}}) \geq K_{S'} $. Therefore, the divisor $E:= f^* \big( K_{S^{\nu}}+\operatorname{Diff}_{S^\nu}(0) \big) - K_{S'} $ is effective.
		
		We may run a $K_{S'}$-MMP over $T$ which terminates either with a good minimal model $ S'' $ of $ S' $ over $T$ and associated Iitaka fibration $ S'' \to Z'' $ over $ T $ or with a Mori fiber space $S''\to Z''$ over $T$; we denote by $ g \colon S' \to S'' $ the induced morphism over $T$. Note that in the first case $ S'' $ is covered by curves $C''$ which are contracted over $T$ and satisfy $K_{S''} \cdot C'' \geq 0$, while in the second case $S''$ is covered by curves $C''$ which are contracted over $T$ and satisfy $ -3 \leq K_{S''} \cdot C'' < 0 $.
		\begin{center}
			\begin{tikzcd}[column sep = large]
				S' \arrow[d, "f" swap] \arrow[rr, "g"] & & S'' \arrow[ddl] \arrow[dr] \\
				S^\nu \arrow[d, "\nu" swap] & & & Z'' \arrow[dll] \\
				S \arrow[r, "h"] & T 
			\end{tikzcd}
		\end{center}
		Pick a point $ s'\in S' \setminus \big( \Supp E \cup \Exc(\nu \circ f) \cup \Exc(g) \big) $ and note that $S$, $S'$ and $S''$ are isomorphic to each other near $s'$. By the above we may find a curve $ C'' \subseteq S'' $ passing through $ s'' := g(s') $ which is contracted over $T$ and satisfies $K_{S''} \cdot C'' \geq -3$. Therefore, there exists a curve $C' \subseteq S'$ passing through $s'$ such that $g(C')=C''$ and whose image $ C := (\nu \circ f) (C') $ is a curve on $ S $ which is contracted by $h$, since $ s' \notin \Exc(\nu \circ f) $ and $ C'' $ is contracted over $ T $. By construction and by the projection formula we obtain
		\[ (K_X+S) \cdot C = (K_X+S) \cdot (\nu \circ f)_* C'= (K_{S'}+E) \cdot C' \geq K_{S'} \cdot C' . \]
		Furthermore, we may write $ K_{S'} = g^* K_{S''} + F $, where $ F $ is an effective $ g $-exceptional divisor. Note that $ C' \nsubseteq \Supp F $, since $ s' \notin \Exc(g) $. Thus, by the projection formula we obtain
		$$ K_{S'} \cdot C' = ( g^* K_{S''} + F ) \cdot C' \geq g^* K_{S''} \cdot C = K_{S''} \cdot C'' \geq -3 . $$
		It follows that the curve $ C $ has the desired properties.
	\end{proof}
	
	\begin{prop}\label{prop:bddci}
		Let $n$ be a non-negative integer and let $\varepsilon$ be a positive real number. Then there exists a positive integer $I$ depending only on $n$, $\varepsilon$ and satisfying the following.
		
		Assume that $(X/Z,B+M)$ is a $3$-dimensional $\Qq$-factorial klt g-pair with data $X'\xto X\to Z$ and $M'$ such that
		\begin{enumerate}
			\item there are exactly $ n $ exceptional divisorial valuations $ E $ over $X$ with $  a(E,X,B+M) \le 1 $, and
			
			\item $ a(E,X,B+M)\ge 1+\varepsilon$ for any exceptional divisorial valuation $E$ over $ X $ with $ a(E,X,B+M)>1$.
		\end{enumerate}
		Then $ID$ is Cartier for any Weil divisor $D$ on $X$.
	\end{prop}
	
	\begin{proof}
		We will prove the statement by induction on $n$.
		
		Assume first that $n=0$ and let $ (X,B+M) $ be a $ \Q $-factorial terminal g-pair such that $ a(E,X,B+M)\ge 1+\varepsilon$ for any exceptional divisorial valuation $E$ over $ X $. Then $(X,0)$ is terminal. Let $p$ be the \emph{index} of $X$, i.e., $ p $ is the smallest positive integer such that $pK_X$ is Cartier, and let $ D $ be a Weil divisor on $ X $. By \cite[Corollary 5.2]{Kaw88}, $pD$ is Cartier. Since there is nothing to prove if $ p = 1 $, we assume that $ p > 1 $. By \cite[Appendix by Y.\ Kawamata, Theorem]{Sho92} we may find an exceptional divisorial valuation $E$ over $ X $ such that $ a(E,X,0)=1+\frac{1}{p}$, and by assumption we infer that $p \leq \frac{1}{\varepsilon}$. We conclude that $I:=\lfloor \frac{1}{\varepsilon}\rfloor!$ has the required property.
		
		Assume now that the statement holds for all integers $ k $ with $ 0 \le k \le n-1$ and let $ (X,B+M) $ be a $\Q$-factorial klt g-pair such that there are exactly $ n \geq 1 $ exceptional divisorial valuations $ E $ over $X$ with $  a(E,X,B+M) \le 1 $ and $ a(E,X,B+M)\ge 1+\varepsilon$ for any exceptional divisorial valuation $E$ over $ X $ with $ a(E,X,B+M)>1$. Let $D$ be a Weil divisor on $X$. To prove the statement, we distinguish two cases: either all these $ n $ exceptional divisorial valuations $E$ over $ X $ have log discrepancy $  a(E,X,B+M) = 1 $ or at least one of them has log discrepancy $  a(E,X,B+M) \in (0,1) $. We treat these cases separately below.
		
		\medskip
		
		\textbf{Case I}: There exists an exceptional divisorial valuation $E$ over $ X $ with log discrepancy $ a:= a(E,X,B+M) \in (0,1) $.
		
		By \cite[Lemmata 4.5 and 4.6]{BZ16} there exist a $ \Q $-factorial klt g-pair $ (Y,\Delta+\Xi) $ with data $ X' \overset{g}{\to} Y \to Z $ and $ M' $ and an extremal contraction $ h \colon Y \to X$ with exceptional prime divisor $ E $ such that
		\begin{equation} \label{eq:univ_1}
			K_Y+\Delta+\Xi \sim_\R h^* (K_X+B+M) .
		\end{equation}
		Here, $ \Delta = B_Y + (1-a) E $, where $ B_Y $ is the strict transform of $ B $ on $ Y $. Possibly replacing $X'$ with a higher model, we may assume that $f = h \circ g$ and that $ g $ is a log resolution of $ (Y,\Delta) $. Also, we denote by $ E' $ the strict transform of $E$ on $X'$ and we note that $ E' $ is an $ f $-exceptional divisor. 
		
		Next, we may write
		\begin{equation} \label{eq:univ_2}
			h^*D \sim_\Q D_Y+qE \ \text{ for some } q \in \Q_{\geq 0} ,
		\end{equation}
		where $D_Y$ is the strict transform of $D$ on $Y$, and
		\begin{equation} \label{eq:univ_3}
			h^*M \sim_\R \Xi+rE \ \text{ for some } r \in \R_{\ge 0} .
		\end{equation}
		Indeed, by \cite[Remark 4.2(3)]{BZ16} and since $ f = h \circ g $, we know that $ f^* M \sim_\R M' + E_f $, where $ E_f $ is an effective $ f $-exceptional $\Rr$-divisor which can be expressed as  $ E_f = \sum q_k E_k' + r E' $, where $ q_k , \, r \in \R_{\geq 0} $ and the $ E_k' $ are $ g $-exceptional prime divisors. We obtain \eqref{eq:univ_3} pushing this relation down to $ Y $ by $ g $. Now, by applying the inductive hypothesis to the g-pair $ (Y,\Delta+\Xi) $, we deduce that there exists a positive integer $I_1'$ depending only on $n$, $\varepsilon$ and such that $I_1'D_Y$ and $I_1'E$ are Cartier.
		
		\begin{claim} \label{claim:univ}
			There exists a positive integer $I_1$ divisible by $I_1'$ and depending on $n$ and $\varepsilon$ such that $q I_1 \in I_1' \Z$.
		\end{claim}
		
		\noindent Grant this for the time being. Then $h^*(I_1D) \sim_\Q I_1D_Y+(qI_1)E$ is Cartier, and thus $I_1D$ itself is Cartier by the Cone theorem, see \cite[Theorem 3.25(4)]{KM98}. Hence, $I_1$ has the required property. It remains to prove Claim \ref{claim:univ} in order to complete the proof in this case.
		
		\begin{proof}[Proof of Claim \ref{claim:univ}]
			First, we show that $a \geq \varepsilon$. Indeed, we may write
			$$ K_{X'}+B'+M' \sim_\R f^*(K_X+B+M) $$ 
			for some $ \R $-divisor $ B' $ on $ X' $. Note that $ \mult_{E'} B' = 1-a $, where $ E' $ is the strict transform of the $ h $-exceptional divisor $E$ on $X'$. If $ \phi \colon W \to X'$ is the blow-up of $X'$ along a general curve $ \gamma $ on $E'$ and if $E_W$ is the irreducible component of the exceptional divisor of $ \phi $ which dominates $ \gamma $, then by \cite[Lemma 2.29]{KM98} and by construction we obtain
			\begin{align*}
				a(E_W,X,B+M) &=  a(E_W,X',B'+M') =  a(E_W,X',B') \\
				&= 2-(1-a) = 1+a .
			\end{align*} 
			Since $ a(E_W,X,B+M) = 1+a > 1 $, we infer that $ a \geq \varepsilon $.
			
			Furthermore, by Lemma \ref{lem:length} there exists a curve $ C \subseteq E $ which is contracted by $ h $ and such that $ (K_Y+E) \cdot C \geq -3$. Note also that $ E \cdot C < 0 $, see the proof of \cite[Proposition 2.5]{KM98}.
			Moreover, \eqref{eq:univ_2} and \eqref{eq:univ_3} yield 
			$$ q= - \frac{D_Y \cdot C}{E \cdot C} \ \text{ and } \ \Xi \cdot C=-rE \cdot C \geq 0 . $$
			Recall now that $ K_Y+ B_Y + (1-a) E +\Xi \sim_\R h^* (K_X+B+M) $, see \eqref{eq:univ_1}. 
			Since $C \nsubseteq \Supp B_Y$, we have 
			$(K_Y+B_Y+E+\Xi)\cdot C \geq -3$, and since $ a \geq \varepsilon > 0 $, we obtain
			$$ -E\cdot C=\frac{-(K_Y+B_Y+E+\Xi)\cdot C}{a}\le\frac{3}{\varepsilon} . $$
			Since
			$$q=\frac{D_Y\cdot C}{-E\cdot C}=\frac{I_1'(D_Y\cdot C)}{-I_1'(E\cdot C)} , $$
			where $ I_1'(D_Y\cdot C) $ and $ -I_1'(E\cdot C) $ are integers, and since $ \varepsilon \leq a < 1 $, the integer $I_1 := \lf\frac{3I_1'}{\varepsilon}\rf!$ satisfies the required properties.
		\end{proof}
		
		\textbf{Case II}: Each one of the $ n $ exceptional divisorial valuations $E_1, \dots, E_n$ over $ X $ has log discrepancy $  a(E_j,X,B+M) = 1 $, where $ 1 \leq j \leq n $.
		
		Recall that $ a(E,X,B+M)\geq 1+\varepsilon$ for any exceptional divisorial valuation $E$ over $ X $ with $ a(E,X,B+M)>1$. Consider an effective ample over $ Z $ $ \R $-divisor $ H $ on $ X $ with sufficiently small coefficients which contains the center on $ X $ of every valuation $ E_j $ so that $ a \big( E_j,X, (B+H)+M \big) < 1$ for every $ 1 \leq j \leq n $ and any other exceptional divisorial valuation $E$ over $ X $ has $ a \big(E,X, (B+H)+M \big) \geq 1+\frac{\varepsilon}{2}$, see \cite[Lemma 2.5]{Kol13}. Moreover, by \cite[Remark 4.2(2)]{BZ16} the g-pair $\big(X, (B+H)+M \big)$ is klt. We may now repeat verbatim the proof from Case I, working with the g-pair $\big(X, (B+H)+M \big)$ instead, and derive thus an integer $ I_2 $ with the desired properties.
		
		\medskip
		
		Finally, note that the integer $ I_2 $ obtained in Case II is possibly different from the integer $ I_1 $ obtained in Case I. Then the integer $ I := I_1 \cdot I_2 $ satisfies the conclusion of the statement. This concludes the proof.
	\end{proof}
	
	\begin{prop}\label{prop:finitelogdiscrepancy}
		Let $n$ be a non-negative integer, let $\varepsilon$ be a positive real number and let $\Gamma$ be a finite set of non-negative real numbers. Then there exists a finite set $\mathcal{LD}\subseteq[0,1]$ depending only on $n$, $\varepsilon$, $\Gamma$ satisfying the following.
		
		Assume that $(X,B+M)$ is a $3$-dimensional $\Qq$-factorial NQC klt g-pair with data $X'\xto X\to Z$ and $M'=\sum\mu_j M_j'$ such that
		\begin{enumerate}
			\item $B \in \Gamma$,
			
			\item $\mu_j\in\Gamma$ and $M_j'$ is a Cartier divisor which is nef over $ Z $ for any $j$,
			
			\item there are exactly $ n $ exceptional divisorial valuations $ E $ over $X$ with $  a(E,X,B+M) \leq 1 $, and
			
			\item $ a(E,X,B+M)\geq 1+\varepsilon$ for any exceptional divisorial valuation $E$ over $ X $ with $ a(E,X,B+M)>1$.
		\end{enumerate}
		Then for any divisorial valuation $F$ over $ X $ with $a(F,X,B+M)\le1$ it holds that $a(F,X,B+M) \in \mathcal{LD}$.
	\end{prop}
	
	\begin{proof}
		Let $ (X,B+M) $ be a g-pair satisfying all the assumptions of the proposition. By Proposition \ref{prop:bddci} there exists a positive integer $I_0$ depending only on $n$ and $\varepsilon$ such that $I_0 D$ is Cartier for any Weil divisor $D$ on $X$. By \cite[Theorem 1.4]{Chen20} there exist a finite set $\Gamma_1$ of positive real numbers and a finite set $\Gamma_2$ of non-negative rational numbers which depend only on $\Gamma$ so that 
		$$ K_X+B+M=\sum \alpha_i(K_X+B^i+M^i), $$
		for some $\alpha_i \in \Gamma_1$ with $\sum \alpha_i=1$, some $B^i \in \Gamma_2$ and some $(M^i)' = \sum_j \mu_{ij} M_j'$ with $ M^i = f_* (M^i)' $ and $\mu_{ij}\in\Gamma_2$ for any $i,j$, and, additionally, each $(X,B^i+M^i)$ is an lc g-pair. Moreover, we may find a positive integer $I$ such that $I \Gamma_2 \subseteq I_0\Z$; in particular, each $I(K_X+B^i+M^i)$ is Cartier. Since for any divisorial valuation $F$ over $ X $ we now have
		$$ a(F,X,B+M)=\sum \alpha_i \, a(F,X,B^i+M^i) , $$
		we conclude that there exists a subset $\mathcal{LD}$ of the finite set
		$$ \left\{ \frac{1}{I} \sum \alpha_i m_i \mid \alpha_i\in\Gamma_1 , \ m_i\in \Z_{\geq 0} \right\} \cap [0,1] $$
		with the desired properties.
	\end{proof}
	
	\begin{thm}\label{thm:gklt}
		Let $(X_1/Z,B_1+M_1)$ be a $3$-dimensional $\Qq$-factorial NQC klt g-pair. Then any sequence of flips over $Z$ starting from $(X_1/Z,B_1+M_1)$ terminates.
	\end{thm}
	
	\begin{proof}
		Let 
		\begin{center}
			\begin{tikzcd}[column sep = 0.8em, row sep = 1.75em]
				(X_1,B_1+M_1) \arrow[dr] \arrow[rr, dashed, "\pi_1"] && (X_2,B_2+M_2) \arrow[dl] \arrow[dr] \arrow[rr, dashed, "\pi_2"] && (X_3,B_3+M_3) \arrow[dl] \arrow[rr, dashed, "\pi_3"] && \dots \\
				& Z_1 && Z_2
			\end{tikzcd}
		\end{center}
		be a sequence of flips over $ Z $ starting from $ (X_1,B_1+M_1) $. For every $ i \geq 1 $ consider a $ \Q $-factorial terminalization $ h_i' \colon (Y_i', \Delta_i' + \Xi_i') \to (X_i, B_i + M_i) $ of $ (X_i, B_i + M_i) $. To show that the given sequence of flips terminates we proceed in three steps.
		
		\medskip
		
		\textbf{Step 1}: 
		By Proposition \ref{prop:klt_discrep} there exists a positive real number $\varepsilon_0$ such that there are $ \nu_0 \geq 0 $ exceptional divisorial valuations $ E $ over $ X_1 $ with $ a(E,X_1,B_1+M_1)<1+\varepsilon_0$, and $ N_0 $ of them, where $ 0 \leq N_0 \leq \nu_0 $, have log discrepancy $ a(E,X_1,B_1+M_1) \leq 1$. It follows from Lemma \ref{lem:monotonicity} that 
		there exist a positive integer $ k $ and a non-negative integer $N \leq N_0$ such that for every $ i \geq k $ there are exactly $ N $ exceptional divisorial valuations $ E_1, \dots, E_N $ over $X_i$ with log discrepancy $  a(E_j,X_i,B_i+M_i)\le1 $ for every $ 1 \leq j \leq N $; 
		note that $ k $ is chosen as the smallest positive integer with this property. In particular, there are exactly $ \nu $, where $ 0 \leq \nu \leq \nu_0 - N $, exceptional divisorial valuations $ F_1, \dots, F_\nu $ over $X_k$ with $  a(F_s,X_k,B_k+M_k) \in (1,1+\varepsilon_0) $ for every $ 1 \leq s \leq \nu $. Set 
		$$ \varepsilon := \min \big\{  a(F_1,X_k,B_k+M_k) - 1 , \dots, \, a(F_\nu,X_k,B_k+M_k) - 1 , \, \varepsilon_0 \big\} $$
		and note that $ \varepsilon = \varepsilon_0 $ if and only if $ \nu = 0 $. It follows from Lemma \ref{lem:monotonicity} that for every $ i \geq k $ and for every exceptional divisorial valuation $E$ over $X_i$ with $ a(E,X_i,B_i+M_i)>1 $ we actually have $ a(E,X_i,B_i+M_i)\ge1+\varepsilon $. By relabelling the given sequence of flips, we may assume that $ k = 1 $.
		
		\medskip
		
		\textbf{Step 2}:
		By Step 1 and by construction of a $ \Q $-factorial terminalization, each $ h_i' \colon Y_i' \to X_i $ extracts the exceptional divisorial valuations $ E_1, \dots, E_N $ and each $ \Delta_i' $ is given by
		\begin{equation*}
			\Delta_i' = (h_i')^{-1}_* B_i + \sum_{j=1}^N \big( 1 - a(E_j,X_i,B_i+M_i) \big) E_j .
		\end{equation*}
		Hence, each $Y_{i+1}'$ is isomorphic in codimension one to $ Y_i' $. Additionally, 
		by Proposition \ref{prop:finitelogdiscrepancy} there exists a finite set $ \mathcal{V} \subseteq [0,1] $ such that for every $ i \geq 1 $ and for every $ 1 \leq j \leq N $ it holds that $ a(E_j,X_i,B_i+M_i) \in \mathcal{V} $. Thus, by Lemma \ref{lem:monotonicity}, for every $ 1 \leq j \leq N $ we obtain a non-decreasing sequence $ \left\{ a(E_j,X_i,B_i+M_i) \right\}_{i=1}^{+ \infty} $ of elements of the finite set $ \mathcal{V} $, which must therefore stabilize. Hence, there exists a positive integer $ r $ such that 
		$\Delta_{i+1}'$ is the strict transform of $\Delta_i'$ for every $ i \geq r $. By relabelling the given sequence of flips, we may assume that $ r = 1 $.
		
		\medskip
		
		\textbf{Step 3}: By Step 2 we may apply Lemma \ref{lem:lift_klt_term} and we obtain thus a diagram
		\begin{center}
			\begin{tikzcd}[column sep = 0.8em, row sep = large]
				(Y_1,\Delta_1+\Xi_1) \arrow[d, "h_1" swap] \arrow[rr, dashed, "\rho_1"] && (Y_2,\Delta_2+\Xi_2) \arrow[d, "h_2" swap] \arrow[rr, dashed, "\rho_2"] && (Y_3,\Delta_3+\Xi_3) \arrow[d, "h_3" swap] \arrow[rr, dashed, "\rho_3"] && \dots \\ 
				(X_1,B_1+M_1) \arrow[dr] \arrow[rr, dashed, "\pi_1"] && (X_2,B_2+M_2) \arrow[dl] \arrow[dr] \arrow[rr, dashed, "\pi_2"] && (X_3,B_3+M_3) \arrow[dl] \arrow[rr, dashed, "\pi_3"] && \dots \\
				& Z_1 && Z_2
			\end{tikzcd}
		\end{center}
		where the top row yields a sequence of flips over $ Z $ starting from the $ \Q $-factorial NQC terminal g-pair $ (Y_1, \Delta_1 + \Xi_1) $. By Theorem \ref{thm:gterminal} the sequence $ Y_i \dashrightarrow Y_{i+1} $ of flips over $ Z $ terminates, and thus the sequence $ X_i \dashrightarrow X_{i+1} $ of flips over $ Z $ terminates. This concludes the proof.
	\end{proof}
	
	\begin{proof}[Proof of Theorem \ref{thm:terminationforthreefolds}]
		Let 
		\begin{center}
			\begin{tikzcd}[column sep = 0.8em, row sep = 1.75em]
				(X_1,B_1+M_1) \arrow[dr] \arrow[rr, dashed, "\pi_1"] && (X_2,B_2+M_2) \arrow[dl] \arrow[dr] \arrow[rr, dashed, "\pi_2"] && (X_3,B_3+M_3) \arrow[dl] \arrow[rr, dashed, "\pi_3"] && \dots \\
				& Z_1 && Z_2
			\end{tikzcd}
		\end{center}
		be a sequence of flips over $ Z $ starting from the given NQC lc g-pair $ (X_1/Z,B_1+M_1) := (X/Z,B+M) $ of dimension $3$. By Lemma \ref{lem:lift_lc_dlt} we obtain a diagram
		\begin{center}
			\begin{tikzcd}[column sep = 0.8em, row sep = large]
				(Y_1,\Delta_1+\Xi_1) \arrow[d, "h_1" swap] \arrow[rr, dashed, "\rho_1"] && (Y_2,\Delta_2+\Xi_2) \arrow[d, "h_2" swap] \arrow[rr, dashed, "\rho_2"] && (Y_3,\Delta_3+\Xi_3) \arrow[d, "h_3" swap] \arrow[rr, dashed, "\rho_3"] && \dots 
				\\ 
				(X_1,B_1+M_1) \arrow[dr] \arrow[rr, dashed, "\pi_1"] && (X_2,B_2+M_2) \arrow[dl] \arrow[dr] \arrow[rr, dashed, "\pi_2"] && (X_3,B_3+M_3) \arrow[dl] \arrow[rr, dashed, "\pi_3"] && \dots \\
				& Z_1 && Z_2
			\end{tikzcd}
		\end{center}
		where each map $ h_i \colon (Y_i, \Delta_i + \Xi_i) \to (X_i,B_i+M_i) $, $i \geq 1$, is a dlt blow-up of the NQC lc g-pair $ (X_i,B_i+M_i) $ and the top row yields an MMP over $ Z $ for the $3$-dimensional NQC $\Q$-factorial dlt g-pair $ (Y_1,\Delta_1 + \Xi_1) $. Now, to prove the statement, it suffices to show that this $ (K_{Y_1}+\Delta_1+\Xi_1) $-MMP terminates. By relabelling, we may assume that it consists only of flips. Additionally, by the special termination for NQC $\Q$-factorial dlt g-pairs of dimension $3$ and by relabelling again, we may assume that this sequence of flips is also a sequence of flips for the $3$-dimensional NQC $\Q$-factorial klt g-pair $ \big(Y_1,(\Delta_1 - \lfloor \Delta_1 \rfloor) + \Xi_1 \big) $. It follows from Theorem \ref{thm:gklt} that the sequence of $ \big( K_{Y_1} + (\Delta_1 - \lfloor \Delta_1 \rfloor) + \Xi_1 \big) $-flips over $Z$ terminates, which concludes the proof.
	\end{proof}

	\section{Proofs of the Main Results}
	\label{section:term_for_psef}
	
	In this section we prove Theorems \ref{thm:terminationforfourfolds} and \ref{thm:mainconsequence_g}.
	The proof of the latter occupies the largest part of the section and is obtained as follows. First, we derive an analog of \cite[Lemma 3.2]{Bir07} in the context of g-pairs by utilizing Lemma \ref{lem:supplement_g} (see Theorem \ref{thm:mainthm}). In addition, we emphasize that the ACC for lc thresholds for g-pairs \cite[Theorem 1.5]{BZ16}, which plays a crucial role in the proofs of Theorem \ref{thm:mainthm}, \cite[Theorem 1]{Mor18} and \cite[Theorem 1]{HM20}, is invoked here without passing to some open subset of the varieties involved (as is the case in \cite{Bir07}). Afterwards, we generalize \cite[Theorem 1.3]{Bir07} to the setting of g-pairs by combining Lemma \ref{lem:lift_lc_dlt} and Theorem \ref{thm:mainthm}
	(see Corollary \ref{cor:maincor}). Finally, we derive Theorem \ref{thm:mainconsequence_g} directly from Corollary \ref{cor:maincor} due to the fact that the special termination for NQC $\Q$-factorial dlt g-pairs can be reduced to the termination of flips in lower dimensions. Note also that Theorem \ref{thm:mainconsequence_g}, together with \cite[Theorem E]{LT19}, yield a different proof of \cite[Corollary G]{LT19}, which does not depend on \cite{HM20}.
	
	\begin{thm} \label{thm:mainthm}
		Assume the existence of minimal models for smooth varieties of dimension $ d-1 $ and the special termination for NQC lc g-pairs of dimension $ d $.
		
		Let
		\begin{center}
			\begin{tikzcd}[column sep = 0.8em, row sep = 1.75em]
				(X_1,B_1+M_1) \arrow[dr, "\theta_1" swap] \arrow[rr, dashed, "\pi_1"] && (X_2,B_2+M_2) \arrow[dl, "\theta_1^+"] \arrow[dr, "\theta_2" swap] \arrow[rr, dashed, "\pi_2"] && (X_3,B_3+M_3) \arrow[dl, "\theta_2^+"] \arrow[rr, dashed, "\pi_3"] && \dots \\
				& Z_1 && Z_2
			\end{tikzcd}
		\end{center}
		be a sequence of flips over $ Z $ starting from an NQC lc g-pair $ (X_1,B_1+M_1) $ of dimension $ d $, where each g-pair $ (X_i,B_i+M_i) $ in the above sequence of flips comes with data $ X_1' \overset{f_i}{\longrightarrow} X_i \to Z $ and $ M_1' $. Assume that there exist an NQC divisor $ P_1' $ on $ X_1' $ and an effective $ \R $-divisor $ N_1 $ on $ X_1 $ such that $ P_1 + N_1 $ is $ \R $-Cartier and $ -(P_i+N_i) $ is ample over $ Z_i $ for every $ i \geq 1 $, where $ P_i := (f_i)_* P_1' $ and $ N_i $ is the strict transform of $ N_1 $ on $ X_i $. Then the given sequence of flips over $ Z $ terminates.
	\end{thm}
	
	\begin{proof}
		Assume that the given sequence of flips does not terminate. We derive a contradiction by violating the ACC for lc thresholds for g-pairs \cite[Theorem 1.5]{BZ16}.
		
		\medskip
		
		\textbf{Step 1}: Let $ t_1 \geq 0 $ be the lc threshold of $ P_1 + N_1 $ with respect to $ (X_1,B_1+M_1) $. Since $ -(P_i+N_i) $ is ample over $ Z_i $ and since $ \rho(X_i/Z_i) = 1 $, we have $ K_{X_i} + B_i + M_i \equiv_{Z_i} \alpha_i (P_i + N_i) $ for some $ \alpha_i >0 $. Therefore, the given sequence of $ (K_{X_1} + B_1 + M_1) $-flips over $ Z $ is also a sequence of $ \big( K_{X_1} + (B_1 + t_1 N_1) + (M_1 + t_1 P_1) \big) $-flips over $ Z $. By the special termination there exists an integer $ \ell \geq 1 $ such that 
		\begin{equation} \label{eq:mainthm_ST}
			\Exc(\theta_i) \cap \nklt\big(X_i, (B_i + t_1 N_i) + (M_i + t_1 P_i) \big) = \emptyset \ \text{ for every } i \geq \ell .
		\end{equation}
		By relabelling the given sequence of flips, we may assume that $ \ell = 1 $.
		
		\medskip
		
		\textbf{Step 2}: Consider a dlt blow-up
		$$ h_1 \colon (Y_1, \Delta_1 + \Xi_1) \to \big(X_1, (B_1+ t_1 N_1) + (M_1 + t_1 P_1) \big) $$
		of the NQC lc g-pair $ \big(X_1, (B_1+ t_1 N_1) + (M_1 + t_1 P_1) \big) $ such that $ f_1 = h_1 \circ g_1 $, where the g-pair $ (Y_1, \Delta_1 + \Xi_1) $ comes with data $ X_1' \overset{g_1}{\longrightarrow} Y_1 \to Z $ and $ \Xi_1' := M_1' + t_1 P_1' $. Set $ L_1 := (g_1)_* M_1' $ and $ Q_1 := (g_1)_* P_1' $ and note that
		$$ \Xi_1 = (g_1)_* \Xi_1' = L_1 + t_1 Q_1 , \
		(h_1)_* L_1 = M_1 \ \text{ and } \ (h_1)_* Q_1 = P _1 . $$
		By Lemma \ref{lem:lift_lc_dlt} there exists some $ (K_{Y_1} + \Delta_1 + \Xi_1) $-MMP with scaling of an ample divisor over $ Z_1 $ which terminates with a dlt blow-up 
		$$ h_2 \colon (Y_2, \Delta_2 + \Xi_2) \to \big(X_2, (B_2 + t_1 N_2) + (M_2 + t_1 P_2) \big) $$
		of the NQC lc g-pair 
		$ \big(X_2, (B_2 + t_1 N_2) + (M_2 + t_1 P_2) \big) $ such that $ f_2 = h_2 \circ g_2 $, where the g-pair $ (Y_2, \Delta_2 + \Xi_2) $ comes with data $ X_1' \overset{g_2}{\longrightarrow} Y_2 \to Z $ and $ \Xi_1' $. As above, set $ L_2 := (g_2)_* M_1' $ and $ Q_2 := (g_2)_* P_1' $ and note that 
		$$ \Xi_2 = (g_2)_* \Xi_1' = L_2 + t_1 Q_2 , \ (h_2)_* L_2 = M_2 \ \text{ and } \ (h_2)_* Q_2 = P_2 . $$
		Thus, we obtain the following diagram:
		\begin{center}
			\begin{tikzcd}[column sep = 0.3em, row sep = 2em]
				& X_1' \arrow[dl, "g_1" swap] \arrow[dr, "g_2"]\\
				(Y_1,\Delta_1+\Xi_1) \arrow[d, "h_1" swap] \arrow[rr, dashed] && (Y_2,\Delta_2+\Xi_2) \arrow[d, "h_2" swap] \\
				\big(X_1, (B_1+ t_1 N_1) + (M_1 + t_1 P_1) \big) \arrow[dr, "\theta_1" swap] \arrow[rr, dashed, "\pi_1"] 
				&& \big(X_2, (B_2+ t_1 N_2) + (M_2 + t_1 P_2) \big) \arrow[dl, "\theta_1^+"] \\
				& Z_1
			\end{tikzcd}
		\end{center}
		By Lemma \ref{lem:supplement_g}(i) this $ (K_{Y_1} + \Delta_1 + \Xi_1) $-MMP with scaling over $ Z_1 $ consists only of flips
		and it is also a sequence of $ (K_{Y_1} + \Delta_1^{<1} + \Xi_1)$-flips over $ Z_1 $, where the g-pair $ (Y_1, \Delta_1^{<1} + \Xi_1) $ is $ \Q $-factorial klt. Additionally, we will demonstrate below that Lemma \ref{lem:supplement_g}(ii) may also be applied in this setting.
		
		To this end, if $ E $ denotes the sum of the $ h_1 $-exceptional prime divisors, then by construction of a dlt blow-up we know that
		\begin{equation} \label{eq:mainthm_1}
			\Delta_1 = (h_1)_*^{-1} \left( (B_1 + t_1 N_1)^{<1} \right) + (h_1)_*^{-1} \left( (B_1 + t_1 N_1)^{=1} \right) + E . 
		\end{equation}
		We define now three effective $ \R $-divisors $ \Gamma_1 $, $ \Delta_1^{\prime} $ and $ \Delta_1^{\prime\prime} $ on $ Y_1 $ such that 
		\begin{equation}\label{eq:mainthm_2}
			\Delta_1^{<1} = \Delta_1^{\prime} + t_1 \Delta_1^{\prime\prime}
			\quad \text{ and } \quad 
			(h_1)_*^{-1} N_1 = \Gamma_1 + \Delta_1^{\prime\prime} 
		\end{equation}
		and $ \Supp \Gamma_1 \subseteq \Supp \Delta_1^{=1} $. We write $ B_1 = \sum \beta_k G_k $ and $ N_1 = \sum \nu_k G_k $, where $ \beta_k \in [0,1] $ and $ \nu_k \in [0, + \infty) $, so that 
		$ B_1 + t_1 N_1 = \sum (\beta_k + t_1 \nu_k) G_k$, where $ \beta_k + t_1 \nu_k \in [0,1] $ by construction. We set
		\begin{gather*}
			\Delta_1^{\prime} := \sum_{k: \beta_k + t_1 \nu_k < 1} \beta_k (h_1)_*^{-1} G_k \ \text{ so that } \ \Supp \big( (h_1)_* \Delta_1^{\prime} \big) \subseteq \Supp B_1 , \\[0.25em] 
			\Delta_1^{\prime\prime} := \sum_{k: \beta_k + t_1 \nu_k < 1} \nu_k (h_1)_*^{-1} G_k \ \text{ so that } \ \Supp \big( (h_1)_* \Delta_1^{\prime\prime} \big) \subseteq \Supp N_1 , \text{ and } \\[0.25em] 
			\Gamma_1 := \sum_{k: \beta_k + t_1 \nu_k = 1} \nu_k (h_1)_*^{-1} G_k  \ \text{ so that } \ \Supp \big( (h_1)_* \Gamma_1 \big) \subseteq \Supp N_1 .
		\end{gather*}
		Using \eqref{eq:mainthm_1}, it is easy to see that \eqref{eq:mainthm_2} holds and we also note that
		\[ \Supp \Gamma_1 \subseteq \Supp \Big( (h_1)_*^{-1} \big( (B_1 + t_1 N_1)^{=1} \big) \Big) \subseteq \Supp \Delta_1^{=1} . \]
		Since $ f_1 = h_1 \circ g_1 $, $ P_1 = (f_1)_* P_1' $ and $ Q_1 = (g_1)_* P_1' $, by \eqref{eq:mainthm_1} and \eqref{eq:mainthm_2} we may write
		$$ (h_1)^*(P_1 + N_1) = Q_1 + (h_1)_*^{-1} N_1 + F_1 = Q_1 + \Delta_1^{\prime\prime} + \big( \Gamma_1 + F_1 \big) , $$
		where $ F_1 $ is an $ h_1 $-exceptional $ \R $-divisor and
		$ \Supp \big( \Gamma_1 + F_1 \big) \subseteq \Supp \Delta_1^{=1} $. 
		Therefore, we may apply Lemma \ref{lem:supplement_g}(ii) and deduce that the sequence of $ (K_{Y_1} + \Delta_1 + \Xi_1) $-flips over $ Z_1 $ is also a sequence of $ (Q_1 + \Delta_1^{\prime\prime}) $-flips over $ Z_1 $.
		
		Note that, due to \eqref{eq:mainthm_ST}, we may repeat the above procedure for every NQC lc g-pair $ \big(X_i, (B_i + t_1 N_i) + (M_i + t_1 P_i) \big) $ of the sequence $ X_i \dashrightarrow X_{i+1} $ of flips over $ Z $ in order to produce a diagram as above (case $ i=1 $), starting with the dlt blow-up $ h_i \colon (Y_i, \Delta_i + \Xi_i) \to \big(X_i, (B_i + t_1 N_i) + (M_i + t_1 P_i) \big) $ obtained from the previous repetition (for $ i \geq 2 $). Specifically, each such repetition produces a sequence of $ (K_{Y_i} + \Delta_i + \Xi_i) $-flips over $ Z_i $, which is also a sequence of $ (K_{Y_i} + \Delta_i^{<1} + \Xi_i) $-flips over $ Z_i $ with respect to the $ \Q $-factorial NQC klt g-pair $ (Y_i, \Delta_i^{<1}  + \Xi_i) $ as well as a sequence of $ ( Q_i + \Delta_i^{\prime\prime} ) $-flips over $ Z_i $, where (as in the case $ i=1 $ we have) $ \Delta_i^{<1} = \Delta_i^{\prime} + t_1 \Delta_i^{\prime\prime} $ and $ \Xi_i = L_i + t_1 Q_i $ with $ L_i =  (g_i)_* M_1' $ and $ Q_i = (g_i)_* P_1' $. Moreover, note that each repetition of the procedure in question is clearly compatible with the previous one.
		
		Since we may consider every flip over $ Z_i $ as a flip over $ Z $, we obtain a sequence of 
		$ \big( K_{Y_1} + \Delta_1^{<1} + \Xi_1 = K_{Y_1} + (\Delta_1^{\prime} + t_1 \Delta_1^{\prime\prime}) + (L_1 + t_1 Q_1) \big) $-flips over $ Z $ starting from the $ \Q $-factorial NQC klt g-pair 
		$ (Y_1, \Delta_1^{<1}  + \Xi_1) = \big(Y_1, (\Delta_1^{\prime} + t_1 \Delta_1^{\prime\prime}) + (L_1 + t_1 Q_1) \big) $, 
		which is also a sequence of $ (Q_1 + \Delta_1^{\prime\prime}) $-flips over $ Z $. Note that this sequence of flips does not terminate, since by construction it is also a sequence of $ (K_{Y_1} + \Delta_1 + \Xi_1) $-flips over $ Z $ and if it terminated, then the given sequence of $ (K_{X_1} + B_1 + M_1) $-flips over $ Z $ would also terminate, a contradiction to our assumption. Additionally, by Lemma \ref{lem:singularities_g}, $ (Y_1, \Delta_1^{\prime} + L_1) $ is a $ \Q $-factorial NQC klt g-pair.
		
		\medskip
		
		\textbf{Step 3}: Let $ t_2 $ be the lc threshold of 
		$ Q_1 + \Delta_1^{\prime\prime} $ with respect to 
		$ (Y_1, \Delta_1^{\prime} + L_1) $. By the previous paragraph we deduce that $ t_2 > t_1 $ and that the sequence of 
		$ (K_{Y_1} + (\Delta_1^{\prime} + t_1 \Delta_1^{\prime\prime}) + (L_1 + t_1 Q_1)) $-flips over $ Z $ is also a sequence of 
		$ (K_{Y_1} + (\Delta_1^{\prime} + t_2 \Delta_1^{\prime\prime}) + (L_1 + t_2 Q_1)) $-flips over $ Z $ with respect to the $ \Q $-factorial NQC lc g-pair 
		$ \big(Y_1,(\Delta_1^{\prime} + t_2 \Delta_1^{\prime\prime}) + (L_1 + t_2 Q_1) \big) $. We may now apply the special termination as in Step 1 and continue as in Step 2.
		
		\medskip
		
		\textbf{Step 4}: By repeating Steps 2 and 3, we get a strictly increasing sequence $ \{t_i\}_{i=1}^{\infty} $ of lc thresholds. However, this contradicts \cite[Theorem 1.5]{BZ16}.
	\end{proof}
	
	\begin{rem}
		\label{rem:comments_on_assumptions}
		If in Theorem \ref{thm:mainthm} the given NQC g-pair $(X_1,B_1+M_1)$ is $\Q$-factorial dlt, then it suffices to assume only the special termination for NQC \emph{$\Q$-factorial dlt} g-pairs of dimension $ d $ in order to prove the statement: this is clear from the above proof.
	\end{rem}
	
	\begin{cor} \label{cor:maincor}
		Assume the existence of minimal models for smooth varieties of dimension $ d-1 $ and 
		the special termination for NQC $\Q$-factorial dlt g-pairs of dimension $ d $.
		
		Let $ (X_1,B_1+M_1) $ be an NQC lc g-pair of dimension $ d $. Assume that $ (X_1/Z,B_1+M_1) $ admits an NQC weak Zariski decomposition over $ Z $. Then any sequence of flips over $ Z $ starting from $ (X_1/Z,B_1+M_1) $ terminates.
	\end{cor}
	
	\begin{proof}
		Let 
		\begin{center}
			\begin{tikzcd}[column sep = 0.8em, row sep = 1.75em]
				(X_1,B_1+M_1) \arrow[dr] \arrow[rr, dashed, "\pi_1"] && (X_2,B_2+M_2) \arrow[dl] \arrow[dr] \arrow[rr, dashed, "\pi_2"] && (X_3,B_3+M_3) \arrow[dl] \arrow[rr, dashed, "\pi_3"] && \dots \\
				& Z_1 && Z_2
			\end{tikzcd}
		\end{center}
		be a sequence of flips over $ Z $ starting from the g-pair $ (X_1/Z,B_1+M_1) $ which comes with data $ X_1' \overset{f_1}{\longrightarrow} X_1 \to Z $ and $ M_1' $. By Lemma \ref{lem:lift_lc_dlt} we obtain a diagram
		\begin{center}
			\begin{tikzcd}[column sep = 0.8em, row sep = large]
				(Y_1,\Delta_1+\Xi_1) \arrow[d, "h_1" swap] \arrow[rr, dashed, "\rho_1"] && (Y_2,\Delta_2+\Xi_2) \arrow[d, "h_2" swap] \arrow[rr, dashed, "\rho_2"] && (Y_3,\Delta_3+\Xi_3) \arrow[d, "h_3" swap] \arrow[rr, dashed, "\rho_3"] && \dots 
				\\ 
				(X_1,B_1+M_1) \arrow[dr] \arrow[rr, dashed, "\pi_1"] && (X_2,B_2+M_2) \arrow[dl] \arrow[dr] \arrow[rr, dashed, "\pi_2"] && (X_3,B_3+M_3) \arrow[dl] \arrow[rr, dashed, "\pi_3"] && \dots \\
				& Z_1 && Z_2
			\end{tikzcd}
		\end{center}
		where each map $ h_i \colon (Y_i, \Delta_i + \Xi_i) \to (X_i,B_i+M_i) $, $i \geq 1$, is a dlt blow-up of the NQC lc g-pair $ (X_i,B_i+M_i) $ and the top row yields an MMP over $ Z $ for the NQC $\Q$-factorial dlt g-pair $ (Y_1,\Delta_1 + \Xi_1) $ with data $X_1' \overset{g_1}{\longrightarrow} Y_1 \to Z$ and $M_1'$. Now, to prove the statement, it suffices to show that this $ (K_{Y_1}+\Delta_1+\Xi_1) $-MMP terminates. By relabelling, we may assume that it consists only of flips. Moreover, by assumption and by \cite[Remark 2.11]{LT19}, the g-pair $ (Y_1,\Delta_1 + \Xi_1) $ admits an NQC weak Zariski decomposition over $ Z $, i.e., there exists a projective birational morphism $ \psi \colon W \to Y_1 $ from a normal variety $W$ such that $ \psi^*(K_{Y_1}+\Delta_1+\Xi_1) \equiv_Z P_W + N_W $, where $ P_W $ is an NQC divisor on $ W $ and $ N_W $ is an effective $ \R $-Cartier $ \R $-divisor on $ W $. Replacing $X_1'$ with a higher model, we may additionally assume that $g_1 = \psi \circ \varphi $, where $ \varphi \colon X_1' \to W $.
		Then $ P_1' := \varphi^* P_W $ is an NQC divisor on $ X_1' $. Set $ P_1 := (g_1)_* P_1' = \psi_* P_W $ and $ N_1 := \psi_* N_W $ and note that $ P_1 + N_1 $ is $ \R $-Cartier and satisfies $ K_{Y_1}+\Delta_1+\Xi_1 \equiv_Z P_1 + N_1 $. We conclude by Theorem \ref{thm:mainthm}, taking into account Remark \ref{rem:comments_on_assumptions}.
	\end{proof}
	
	\begin{proof}[Proof of Theorem \ref{thm:mainconsequence_g}]
		By the assumptions of the theorem in lower dimensions, we may assume the special termination for NQC $\Q$-factorial dlt g-pairs of dimension $ d $ (see \cite[Theorem 5.1 and Lemma 5.2]{LMT}). We conclude by Corollary \ref{cor:maincor}. We emphasize again that \cite{LMT} does not depend on \cite{Mor18}.
	\end{proof}
	
	\begin{proof}[Proof of Theorem \ref{thm:terminationforfourfolds}]
		By combining \cite[Theorem 5-1-15]{KMM87}, \cite[Proposition 5.1]{HanLi} and \cite[Theorem E]{LT19}, we infer that the given g-pair $(X/Z,B+M)$ admits an NQC weak Zariski decomposition over $ Z $. We conclude by combining Theorems \ref{thm:terminationforthreefolds} and \ref{thm:mainconsequence_g}.
	\end{proof}

	\bibliographystyle{amsalpha}
	\providecommand{\bysame}{\leavevmode\hbox to3em{\hrulefill}\thinspace}
	\providecommand{\MR}{\relax\ifhmode\unskip\space\fi MR }
	\providecommand{\MRhref}[2]{%
		\href{http://www.ams.org/mathscinet-getitem?mr=#1}{#2}
	}
	\providecommand{\href}[2]{#2}

\end{document}